\documentclass[12pt,letterpaper,reqno]{amsart}

\usepackage{epsfig}
\usepackage{amsmath}
\usepackage{amssymb}
\usepackage{amsthm}
\usepackage{indentfirst}
\usepackage{xspace}
\usepackage{hyperref}
\usepackage{xcolor}
\usepackage{setspace}
\definecolor{darkred}{rgb}{0.5,0.15,0.15}
\hypersetup{colorlinks=true,urlcolor=darkred,linkcolor=darkred,citecolor=darkred}
\usepackage{verbatim}
\usepackage[letterpaper,margin=1in,headheight=15pt]{geometry}
\usepackage{mathpazo}
\usepackage{tikz-cd}
\usepackage{booktabs}
\usepackage{float} 
\definecolor{shadecolor}{rgb}{0.85,0.85,0.85}
\usepackage{bibentry}

\usepackage{bm}

\allowdisplaybreaks

\newtheorem{thm}{Theorem}

\newtheorem{lem}[thm]{Lemma}
\newtheorem{prop}[thm]{Proposition}

\theoremstyle{remark}
\newtheorem{rem}[thm]{Remark}

\theoremstyle{definition}
\newtheorem{defn}[thm]{Definition}

\numberwithin{thm}{section}
\numberwithin{equation}{section}
\numberwithin{figure}{section}

\newcommand{\fa}{{\mathfrak a}}

\newcommand{\cN}{\ensuremath{\mathcal N}}
\newcommand{\cC}{\ensuremath{\mathcal C}}

\newcommand{\cB}{\ensuremath{\mathcal B}}
\newcommand{\cL}{\ensuremath{\mathcal L}}

\newcommand{\cM}{\ensuremath{\mathcal M}}
\newcommand{\cO}{\ensuremath{\mathcal O}}

\newcommand{\cW}{\ensuremath{\mathcal W}}
\newcommand{\cP}{\ensuremath{\mathcal P}}

\newcommand{\R}{\ensuremath{\mathbb R}}
\newcommand{\C}{\ensuremath{\mathbb C}}
\newcommand{\CP}{\ensuremath{\mathbb {CP}}}

\newcommand{\PP}{\ensuremath{\mathbb P}}
\newcommand{\Z}{\ensuremath{\mathbb Z}}

\newcommand{\bfb}{\ensuremath{\mathbf b}}
\newcommand{\bfalpha}{\ensuremath{\bm \alpha}}

\newcommand{\half}{\ensuremath{\frac{1}{2}}}

\newcommand{\cE}{{\mathcal E}}

\newcommand{\hk}{hyperk\"ahler\xspace}

\newcommand{\zbar}{\ensuremath{\overline{z}}}

\newcommand{\Id}{\ensuremath{\mathrm{Id}}}
\newcommand{\eff}{\ensuremath{\mathrm{eff}}}

\newcommand{\tU}{{\widetilde{U}}}

\newcommand{\I}{{\mathrm i}}
\newcommand{\e}{{\mathrm e}}
\newcommand{\de}{\mathrm{d}}

\newcommand{\rmbig}{\mathrm{big}}
\newcommand{\rmsmall}{\mathrm{small}}

\newcommand{\abs}[1]{\lvert#1\rvert}
\newcommand{\norm}[1]{\lVert#1\rVert}

\newcommand{\eps}{\epsilon}

\newcommand{\ti}[1]{\textit{#1}}

\newcommand{\phimodel}{\varphi_{\mathrm{model}}}

\DeclareMathOperator{\Tr}{Tr}

\DeclareMathOperator{\End}{End}
\DeclareMathOperator{\Hom}{Hom}

\DeclareMathOperator{\diag}{diag}
\DeclareMathOperator{\pdeg}{pdeg}

\DeclareMathOperator{\charp}{char}

\begin{document}
\onehalfspacing

\title{From $S^1$-fixed points to $\mathcal{W}$-algebra
representations}
\author{Laura Fredrickson and Andrew Neitzke}

\date{\today}

\begin{abstract}
We study a set $\cM_{K,N}$ parameterizing filtered $SL(K)$-Higgs bundles over $\C\PP^1$ with an 
irregular singularity at $z = \infty$, such that the eigenvalues of the Higgs field
grow like $\abs{\lambda} \sim \abs{{z}^{N/K} \de z}$, where $K$ and $N$ are coprime.
$\cM_{K,N}$ carries a $\C^\times$-action analogous to
the famous $\C^\times$-action introduced by Hitchin on the moduli spaces of Higgs bundles over
compact curves. The construction of this $\C^\times$-action on $\cM_{K,N}$ involves the rotation
automorphism of the base $\C\PP^1$. We classify the fixed points of this $\C^\times$-action,
and exhibit a curious $1$-$1$ correspondence between these fixed points
and certain representations of the vertex algebra $\cW_K$; in particular
we have the relation
$\mu = \frac{1}{12} \left(K - 1 - c_\eff \right)$,
where $\mu$ is a regulated version of the $L^2$ norm of the Higgs field,
and $c_\eff$ is the effective Virasoro central charge of the corresponding
$W$-algebra representation. We also discuss a Bialynicki-Birula-type stratification
of $\cM_{K,N}$, where the strata are labeled by
isomorphism classes of the underlying filtered vector bundles.
\end{abstract}

\maketitle

\setcounter{page}{1}

\bigskip

\tableofcontents

\section{Introduction}\label{sec:introS1}

\subsection{Higgs bundles and singularities}

Recall that a \ti{$SL(K)$-Higgs bundle} over a complex
curve $C$ is a pair $(\cE, \varphi)$, where $\cE$ is a
holomorphic rank $K$ vector bundle with fixed determinant,
and $\varphi$ is a traceless holomorphic section of $\End \cE \otimes K_C$.
Given a compact Riemann surface $C$, there is a moduli space $\cM_K(C)$
parameterizing $SL(K)$-Higgs bundles over $C$ up to equivalence,
first introduced in \cite{hitchin87} when $K=2$.

The $\cM_K(C)$ are geometrically tremendously rich spaces,
but also rather complicated to study explicitly. Part of the reason
for the complication is that to get nontrivial examples one
needs to take $C$ of genus $g \ge 2$, and many of the geometric
phenomena in $\cM_K(C)$ really depend on the moduli of $C$.
If your interest is in Teichm\"uller theory or its higher analogues,
then the fact that $\cM_K(C)$ has a lot to do with $C$ is
a good thing. But if you are interested in other aspects of
$\cM_K(C)$ --- say its \hk structure, its Hodge theory,
its relation to cluster
algebras --- you might want some basic examples where the problems
of $C$ can be set aside.

One way to avoid these difficulties is to consider Higgs bundles
with regular singularities, i.e. allow $\varphi$ to have simple
poles at points of $C$, as in \cite{simpsonnoncompact, yokogawa93, konno93}.
In this case one can get
a nontrivial moduli space even for a genus $0$ curve, e.g. with
$K = 2$ and $4$ regular singularities.
If one is willing to go to $K>2$, then
one can get interesting examples even on a genus $0$ curve with
$3$ regular singularities;
this gets rid of the moduli of $C$ completely since all $3$-tuples
of points are equivalent (but still
leaves the residual awkwardness of choosing a representative.)

There is also another possibility, which is the focus
of this paper: one can study moduli spaces of Higgs bundles with
irregular (aka wild) singularities.
It has been known for some time that many features of
$\cM_K(C)$ have analogues for these kinds of moduli spaces;
in particular, \cite{biquardboalch} shows that one can get \hk moduli
spaces in this way.
Moreover, once we go to the wild setting, there is no obstacle to taking
$C$ to be a \ti{once}-punctured sphere.
This gives rise to a class of examples which
have some right to be called the simplest moduli spaces of 
nonabelian Higgs bundles.

\subsection{The $\cM_{K,N}$}

This paper concerns a family of sets $\cM_{K,N}$,
parameterizing $SL(K)$-Higgs bundles with parabolic
degree zero\footnote{See Definition \ref{def:pdeg}
below for the notion of parabolic degree. We could similarly consider
nonzero parabolic degrees $d$, parameterized by a set
$\cM_{K,N}^d$; but for any $d \in \R$, the map
$(\cE, \varphi) \to (\cE \otimes \cO(\frac{d}{K}), \varphi)$
gives an isomorphism $\cM_{K,N} \simeq \cM_{K,N}^d$.
The harmonic metrics, which we introduce later, 
are similarly modified by the simple overall change
$h \to h \cdot  (1+\abs{z}^2)^{-\frac{d}{K}}$.
Thus, without loss of
generality we may as well just consider $d=0$,
and in the body of the paper we build this into our
definition of $SL(K)$-bundle.}
on $C = \C\PP^1$, with a specific sort of irregular singularity
at $z = \infty$. The singularity condition is specified by an
integer $N > 0$, coprime to $K$;
roughly, it says that the eigenvalues of the Higgs field $\varphi$
behave as
\begin{equation} \label{eq:eigenvalue-condition}
	 \lambda_r \sim \e^{2 \pi \I r/K} z^{N/K} \de z, \qquad \abs{z} \to \infty.
\end{equation}
The precise condition on the Higgs bundles we consider
is given in \S\ref{sec:CKN} below.
It is formulated in the language of \ti{good filtered Higgs bundles}
in the sense of \cite{MochizukiToda}, which we review
in \S\ref{sec:goodfiltered}.

Our assumption that $K$ and $N$ are coprime is used many times
to simplify arguments in this paper. In particular, it implies
that we do not have to worry about stability:
the good filtered Higgs bundles parameterized by $\cM_{K,N}$ are
all automatically stable (Remark \ref{rem:irreducible} below.)

\subsection{Geometric expectations}

It is generally expected
that $\cM_{K,N}$ has most of the same geometric structures
as the usual spaces $\cM_K(C)$.
In this paper, though, we do not develop these structures
in detail. We
construct and study $\cM_{K,N}$ as a \ti{set};
we do not give a careful proof
that it is actually a coarse moduli space. 
Needless to say, we also do not
prove that $\cM_{K,N}$ admits a canonical
\hk structure on its smooth locus,
though this is also widely expected to be true, and should follow
from a small extension of the results of \cite{biquardboalch}.
The nonabelian Hodge correspondence in \cite{biquardboalch} does extend, 
so points of $\cM_{K,N}$ may be thought
of as good filtered bundles, harmonic bundles, or certain bundles equipped with flat connections.
The corresponding wild character varieties are conjecturally those in \cite{BY15}.

$\cM_{K,N}$ maps surjectively to a Hitchin base $\cB_{K,N}$, which is
a linear space of complex dimension $\half (K-1)(N-1)$. We expect
that this map behaves like the usual Hitchin fibration on $\cM_K(C)$,
e.g. that the generic fibers are compact complex tori and all fibers
are compact complex Lagrangian. More specifically the fibers should be
compactified Jacobians\footnotemark\ for the family of spectral curves
parameterized by $\cB_{K,N}$, described in
Proposition \ref{prop:hitchinbase} below. For example, when $K=2$,
this family consists of all curves of the form
\begin{equation}
	y^2 = z^N + P_2(z), \qquad \deg P_2 \le \frac{N}{2}-1.
\end{equation}
\footnotetext{For $SL(K)$-Higgs bundles over a compact curve $C$,
a generic fiber of the Hitchin integrable system
is the Prym subvariety of the Jacobian variety of a spectral curve.
The codimension of the Prym subvariety is equal to the dimension
of $\mathrm{Jac}(C)$. In our situation, the Prym subvariety
coincides with the Jacobian variety, since $\mathrm{Jac}(\CP^1)$
is a point --- fixing the degree $\deg(\cE)=\deg(\wedge^K \cE)=d$
automatically fixes the holomorphic structure
of the determinant line bundle $\wedge^K \cE=\cO(d)$.}

\subsection{The stratification}

The Higgs bundles we consider are in particular
\ti{filtered} bundles over $(\CP^1, \{ \infty \})$,
and decompose as direct sums of filtered line bundles:
\begin{equation} \label{eq:decomp-intro}
	\cE \simeq \bigoplus_{i=1}^K \cO(\alpha_i).
\end{equation}
Here the $\alpha_i \in \R$ are the parabolic degrees
of the summands. The existence of the decomposition
\eqref{eq:decomp-intro} comes from a mild generalization of the usual
Grothendieck lemma for ordinary holomorphic vector
bundles over $\C\PP^1$, Lemma \ref{lem:parabolic-grothendieck} below. 
(Ordinary holomorphic vector bundles over $\CP^1$ can be thought of
as special cases of filtered bundles over
$(\CP^1, \{ \infty \})$, for which the $\alpha_i \in \Z$.)

$\cM_{K,N}$ is stratified by the types occurring in the
decomposition  \eqref{eq:decomp-intro}:
the $\alpha_i$ mod $1$
are fixed (and all distinct) while the integer parts
can change as one moves around $\cM_{K,N}$.
The strata can be conveniently labeled by
cyclic $K$-partitions of $N$, i.e. tuples
$\bfb = (b_1, \dots, b_K)$ with $b_i \in \Z_{\ge 0}$, $\sum b_i = N$,
up to cyclic permutation:
\begin{equation} \label{eq:strat-intro}
	\cM_{K,N} = \bigsqcup_{[\bfb]} \cM_{K,N}^{[\bfb]}.
\end{equation}
The numbers $b_i - \frac{N}{K}$
are the successive
differences $\alpha_i - \alpha_{i+1}$.

\subsection{The $\C^\times$-action}

The stratification \eqref{eq:strat-intro} is
a Bialynicki-Birula-type stratification associated to a certain
$\C^\times$-action on $\cM_{K,N}$, as follows.

Recall that one of the main tools in the
study of the moduli spaces $\cM_K(C)$
is a $\C^\times$-action thereon,
\begin{equation} \label{eq:Cstar-action-usual-intro}
	(\cE, \varphi) \mapsto (\cE, \e^{t} \varphi), \qquad t \in \C / 2 \pi \I \Z.
\end{equation}
Using this $\C^\times$-action one can reduce
questions about $\cM_K(C)$ to questions localized
to an infinitesimal neighborhood of the fixed locus. This
still leaves the difficulty of understanding that
fixed locus concretely, which involves
diverse components with diverse dimensions and interesting topology.

The action \eqref{eq:Cstar-action-usual-intro}
cannot be taken directly over to $\cM_{K,N}$, morally because it does not preserve
the condition  \eqref{eq:eigenvalue-condition}. However, there is a way of fixing this
problem: we combine the rescaling \eqref{eq:Cstar-action-usual-intro}
with an automorphism of $\CP^1$ fixing $z = \infty$,
\begin{equation}
  \rho_t(z) = \e^{-\frac{K t}{K+N}} z, \qquad t \in \C / 2 \pi \I (K+N) \Z,
\end{equation}
to make
\begin{equation} \label{eq:new-cstar-action-intro}
  (\cE, \varphi) \mapsto (\rho_t^* \cE, \e^t \rho_t^* \varphi), \qquad t \in \C / 2 \pi \I (K+N) \Z.
\end{equation}
Thus we get a $\C^\times$-action on $\cM_{K,N}$,
analogous to the usual one on $\cM_K(C)$.
It preserves the stratification \eqref{eq:strat-intro},
and has a single fixed point in each stratum:
see Proposition \ref{prop:classification-fixed-points} below.

\subsection{The fixed points}

As we have just explained, the $\C^\times$-action \eqref{eq:new-cstar-action-intro}
on $\cM_{K,N}$ has finitely many fixed points, labeled
by cyclic $K$-partitions $[\bfb]$ of $N$.
In particular, all fixed components are $0$-dimensional.

The fixed points can be described explicitly: they have
representative Higgs bundles of the form
\begin{equation} \label{eq:introfixedpoints}
  \cE_\bfb = \bigoplus_{i=1}^K \cO(\alpha_i), \qquad \varphi_\bfb = \begin{pmatrix}
               0 & z^{b_1} & &\\ & &  \ddots & \\
                & & & z^{b_{K-1}}\\
                z^{b_K} & & &
              \end{pmatrix} \de z,
\end{equation}
where
$b_i - \frac{N}{K} = \alpha_{i}-\alpha_{i+1}$.
Moreover, for these Higgs bundles the Hitchin equations
can be reduced to a coupled system of
ODE in the radial coordinate; essentially
this is because the $S^1 \subset \C^\times$ preserves the
form of the Hitchin equations, and its action on $\cM_{K,N}$
involves a rotation in the plane.
The ODE in question is a version of the Toda lattice,
\eqref{eq:toda} below.

Thus we obtain a very concrete description of the solutions
of Hitchin equations which arise at the $\C^\times$-fixed
points in $\cM_{K,N}$.
This allows us to compute some invariants of
the solutions in closed form. In particular, we consider
a regulated version of the $L^2$ norm of the Higgs field:
\begin{equation} \label{eq:regulated-norm-intro}
\mu = \frac{\I}{\pi} \int \mathrm{Tr} \left( \varphi \wedge \varphi^{\dagger_h} - \Id |z|^{2N/K} \, \de z \de \zbar \right).
\end{equation}
This regulated
norm turns out to be a rational number, computable in terms of the
parabolic degrees $\alpha_i$ by
\begin{equation}
	\mu = \frac{K}{K+N} \norm{\bfalpha}^2.
\end{equation}
This is analogous to the case of $\cM_2(C)$ \cite{hitchin87}, where
the $L^2$ norm of the Higgs field at the $\C^\times$-fixed points
turns out to be half-integer, and (when nonzero) determined
by the degree of a certain line subbundle of $\cE$.

\subsection{The central fiber}\label{sec:centralfiber}
All of the $\C^\times$-fixed points 
belong to the fiber of $\cM_{K,N}$ lying
over the spectral curve $y^K = z^N$.  We call this fiber the \ti{central fiber}, and we expect (but do not prove) that it is the
compactified Jacobian of the curve $y^K = z^N$.

The compactified Jacobian of the curve $y^K=z^N$ has been 
studied in \cite{Piontkowski} in the algebraic language
of rank-1 torsion-free 
$R=\C[y,z]/(y^K-z^N)$-modules.\footnote{Equivalently, points
of the compactified Jacobian represent torsion-free (but not necessarily locally free) sheaves of rank $1$ and degree $0$ over the spectral curve $y^K=z^N$.}
Proposition 5 of \cite{Piontkowski} 
describes the $R$-modules which appear to correspond to our $\C^\times$-fixed points.
In \S\ref{sec:otherinterpretations} we spell out a dictionary
between these objects and our fixed points,
proposed to the first author by Eugene Gorsky.

\subsection{$W$-algebra minimal models} \label{sec:introW}

Now we come to a surprising fact, which was the initial
motivation for writing this paper:
the rational numbers $\mu$ which we have associated to the
fixed points by \eqref{eq:regulated-norm-intro}
turn out to have another, quite different meaning.

A little background (see \S\ref{sec:Walgebra} for more):
for every $(K,N)$ there is a well-known
vertex algebra $\cW_{K}$ and a certain
package $\Lambda_{K,N}$ of representations of $\cW_{K}$, called a
\ti{minimal model};
for each representation there is a real number $c_\eff$,
the \ti{effective Virasoro central charge}.
Our observation is that there is a canonical correspondence
between the $\C^\times$-fixed points in $\cM_{K,N}$ and the
representations in $\Lambda_{K,N}$, under which
$\mu$ is very simply related to $c_\eff$:
\begin{equation} \label{eq:muc-relation}
 \mu = \frac{1}{12} \left(K - 1 - c_\eff \right).
\end{equation}

\subsection{Argyres-Douglas theories} \label{sec:argyres-douglas}

The formula \eqref{eq:muc-relation} is puzzling.
Why should
$\cW_{K}$ and $\cM_{K,N}$ have anything to do with
one another?

One physics context where $\cM_{K,N}$ arises
was described in \cite{GMNhitchin} building on
\cite{Witten:1997sc,Cherkis:2000cj,Cherkis:2000ft}
(see also \cite{Neitzke:2014cja} for a review): $\cM_{K,N}$
is the moduli space of a certain four-dimensional
$\cN=2$ supersymmetric
quantum field theory, compactified to three dimensions on $S^1$.
The field theory in question is
known as the \ti{Argyres-Douglas theory
of type $(A_{K-1}, A_{N-1})$} \cite{GMNhitchin,Cecotti:2010fi,Xie2012a}.

In this context the vertex algebra $\cW_{K}$ also makes an
appearance. Indeed, it was very recently
shown in \cite{vertexalgebras} that every $\cN=2$
supersymmetric field theory has an associated vertex algebra.
Shortly afterward, in \cite{cordovashao},
it was proposed that the vertex algebra
for the Argyres-Douglas theory of type $(A_{K-1}, A_{N-1})$
should be $\cW_{K}$, and the relevant representations should
be those appearing in $\Lambda_{K,N}$.
Considerable circumstantial evidence in favor of this proposal
has subsequently been given in
\cite{Song2015,Cordova:2016uwk,Cordova:2017mhb,Cordova:2017ohl}.

So at least $\cM_{K,N}$, $\cW_{K}$ and $\Lambda_{K,N}$ all arise
in the context of the Argyres-Douglas theory. One may hope
that the explanation of the formula \eqref{eq:muc-relation}
will also ultimately be found in that theory.
So far we have not found such an explanation;
the correspondence seems to us to be the tip
of an iceberg of unknown size.

After the main results of this paper had been found,
they were used (in the case $K=2$)
in the work \cite{FPYY}, which concerns
a supersymmetric index in Argyres-Douglas theories.
That work also significantly
broadens the scope of the correspondence, by exhibiting several other
examples of Higgs bundle moduli spaces, their corresponding
vertex algebras, and matchings between fixed components
and vertex algebra representations; this includes
examples where some fixed components have nonzero dimension.
The results of this paper were also used very recently
in \cite{Neitzke2017}, which concerns different 
supersymmetric indices (``line defect Schur indices'') in Argyres-Douglas
theories, which are linear combinations of characters of 
representations in $\Lambda_{K,N}$
as previously observed in \cite{Cordova:2016uwk}.

\subsection{The case of $\cM_{2,1}$ and ends of the moduli space $\cM_K(C)$}

The simplest example of our construction is the set $\cM_{2,1}$, which
has only one element. A representative good
filtered Higgs bundle is
\begin{equation}\label{eq:fid}
 \cE \simeq \cO\left( \frac14\right) \oplus \cO\left(-\frac14\right), \qquad \varphi = \begin{pmatrix} 0 & z\\ 1 & 0 \end{pmatrix} \de z.
\end{equation}
The corresponding harmonic metric on $\cE$ is
\begin{equation}\label{eq:fidh}
 h= \begin{pmatrix} |z|^{-1/2}\e^{-u} & \\ & |z|^{1/2} \e^{u} \end{pmatrix},
\end{equation}
where $u=u(|z|)$ is the solution of the Painleve III ODE
\begin{equation}
 \left( \frac{\de^2}{\de|z|^2} + \frac{1}{|z|} \frac{\de}{\de |z|} \right) u = 8 |z|
  \sinh(2 u)
\end{equation}
with $u \sim - \frac{1}{2} \log |z|$ as $z \to 0$ (so that $h$ is smooth)
and $u \rightarrow 0$ as $|z| \rightarrow \infty$ \cite{McCoy-Tracy-Wu}.

This ``fiducial'' solution of Hitchin's equations
on $\C=\CP^1-\{\infty\}$ appeared in \cite{Cecotti:1991me,GMNhitchin}. It
plays a crucial role in Mazzeo-Swoboda-Weiss-Witt's
description of the ``generic ends'' of the moduli space $\cM_2(C)$ \cite{MSWW14},
i.e. the ends corresponding to Higgs bundles for which
the eigenvalues of the Higgs field $\varphi$ have only simple ramification,
$(\lambda_1-\lambda_2)^2 \sim z \de z^2$. Roughly,
as one follows a generic ray toward infinity in $\cM_K(C)$, the harmonic metric on
a sufficiently small disc around a ramification point approaches
the fiducial solution \eqref{eq:fidh}.

In the extension of \cite{MSWW14}
to \ti{non}-generic ends of $\cM_2(C)$,
the points of $\cM_{2,N}$ are expected to play a similar role:
indeed, at ramification points with
$(\lambda_1-\lambda_2)^2 \sim z^N \de z^2$,
some of the relevant model solutions lie in $\cM_{2,N}$.
Similarly the spaces $\cM_{K,N}$ should be relevant for the
extension of \cite{MSWW14} to the higher-rank spaces $\cM_K(C)$.

\subsection{The case of $\cM_{2,3}$} In the final section of this
paper, \S\ref{sec:M23}, we discuss in some detail the next simplest 
example, namely the case of $\cM_{2,3}$. In this case there are just two
strata, $\cM_{2,3}^\rmbig = \cM_{2,3}^{[(2,1)]} \simeq \C^2$ and $\cM_{2,3}^\rmsmall = \cM_{2,3}^{[(3,0)]} \simeq \C$,
and we exhibit explicitly Higgs bundles representing each point
of $\cM_{2,3}$.
$\cM_{2,3}^\rmbig$ consists of $[(\cE, \varphi)]$ where
$\cE \simeq \cO(\frac14) \oplus \cO(-\frac14)$,
while $\cM_{2,3}^\rmsmall$ consists of
$[(\cE, \varphi)]$ where
$\cE \simeq \cO(-\frac34) \oplus \cO(\frac34)$.

We also exhibit directly that the fibers of the Hitchin
map are compact complex tori, except for the central fiber
which is a cuspidal cubic curve.
Each fiber meets $\cM_{2,3}^\rmsmall$ in exactly
one point.
The two $\C^\times$-fixed points, with
$\mu = \frac{9}{20}$ and $\mu = \frac{1}{20}$, correspond to the two
Virasoro representations in the $(2,5)$ Virasoro minimal model,
with $c_\eff = -\frac{22}{5}$, $c_\eff = \frac{2}{5}$ respectively.

\subsection*{Acknowledgements}

We are happy to
thank Philip Boalch, Clay C\'ordova, Eugene Gorsky, Steven Rayan, Szil\'ard Szab\'o
and Fei Yan for useful discussions.
LF acknowledges support from U.S. National Science Foundation grants DMS 1107452, 1107263, 1107367 ``RNMS: GEometric structures And Representation varieties'' (the GEAR Network).
AN's work was supported by National Science Foundation
award 1151693 and by a Simons Fellowship in Mathematics.

\section{The set \texorpdfstring{$\mathcal{M}_{K,N}$}{M(K,N)}}
\label{sec:2}

Fix $K$ and $N$ coprime. In this section we will define a
set $\cM_{K,N}$, parameterizing $SL(K)$-Higgs bundles $(\cE,\varphi)$
on $\CP^1$, with a singularity at $z = \infty$,
obeying a growth condition:
the eigenvalues of the Higgs field $\varphi$
behave as $\lambda_r \sim \e^{2 \pi \I r/K} z^{N/K} \de z$ as $\abs{z} \to \infty$.

Moduli spaces of meromorphic Higgs bundles with poles of arbitrary
order have been considered before, in particular in
\cite{biquardboalch}. That reference includes the technical condition
that the polar part of the Higgs field can be diagonalized
in a neighborhood of each singularity.
We will need to remove this assumption, since in our case
the eigenvalues are ramified in a neighborhood of $z = \infty$.
Thus we use instead the notion of ``good filtered Higgs bundle''
as in \cite{MochizukiToda} \S2.1.1; this allows Higgs bundles
for which the polar part diagonalizes
only after passing to some local ramified cover.

\subsection{Filtered bundles}

We begin with some preliminaries about filtered bundles and filtered
Higgs bundles. For more on this material see e.g. \cite{MochizukiToda}.

Let $C$ be a compact complex curve and let $D \subset C$ be a finite subset.
Let $\cO_C(*D)$ be the sheaf of algebras of rational functions
with poles along $D$, i.e. the localization of $\cO_C$ along $D$.

\begin{defn}\label{defn:filtered}
A \emph{filtered rank $K$ bundle} on $(C, D)$ is a
locally free $\mathcal{O}_C(*D)$-module $\cE$
of finite rank $K$, with an increasing filtration
by locally free $\mathcal{O}_C$-submodules $(\mathcal{P}_\alpha\mathcal{E})_{\alpha \in \R}$ such that
\begin{itemize}
 \item $\mathcal{P}_\alpha \mathcal{E} \vert_{C-D} = \cE \vert_{C-D}$.
 \item $\mathcal{P}_\alpha \mathcal{E} = \underset{\beta>\alpha}{\bigcap}
 \mathcal{P}_\beta \mathcal{E}$.
 \item If $x$ is a local coordinate on a neighborhood $U$ of $p \in D$, then
 $\mathcal{P}_{\alpha-1}\mathcal{E} \vert_{U} = x \mathcal{P}_\alpha \mathcal{E} \vert_{U}$.
 (Thus the filtration $\cP_\bullet \cE$ is determined by the $\cP_\alpha \cE$ with $\alpha \in [0,1)$.)
\end{itemize}
\end{defn}

\begin{defn} Suppose $\cE$ is a filtered bundle over
$(C,D)$.
Given a point $p \in D$, open set $U \subset C$ with $U \cap D = \{p\}$,
and a section $s$ of $\cE$ over $U$,
we define the \ti{order} of $s$ at $p$ to be
\begin{equation}
  \nu_p(s) = \inf \, \{\alpha: s \in \cP_\alpha(\cE)\}.
 \end{equation}
\end{defn}

Direct sums and tensor products of filtered bundles have natural filtered
structures:
\begin{align}
	\cP_\alpha(\cE \oplus \cE') &= \cP_\alpha \cE \oplus \cP_\alpha \cE', \\
	\cP_\alpha (\cE \otimes \cE') &= \sum_{\beta+\gamma = \alpha} \cP_\beta \cE \otimes \cP_\gamma \cE'.
\end{align}

Exterior powers of filtered
bundles also get filtered structures as
subbundles of tensor powers. We have
$\nu_p(s \wedge s') = \nu_p(s \otimes s' - s' \otimes s) \leq \nu_p(s) + \nu_p(s')$,
where in general the inequality can be strict because of cancellations, 
e.g. in the most extreme case,
if $s = f s'$ then $\nu_p(s \wedge s') = \nu_p(0) = - \infty$.
It will be useful later to have a sufficient condition which guarantees that
such a cancellation does not occur:

\begin{lem} \label{lem:filtered-independence}
If $s$ and $s'$ are sections of a filtered bundle $\cE$ over $(C,D)$,
and $\nu_p(s) - \nu_p(s') \notin \Z$, then
$\nu_p(s \wedge s') = \nu_p(s) + \nu_p(s')$.
\end{lem}

\begin{proof} The proof is motivated by the equivalence between filtered bundles and 
parabolic bundles.
Let $0 \leq \alpha_1 \leq \alpha_2 \leq \cdots \leq \alpha_K <1$ be the weights (with multiplicity) where $\cP_{\alpha_i} \cE \neq \cP_{\alpha_i - \eps}\cE$. Let $e_i$ be a local basis of sections in which $\nu_p(e_i) = -\alpha_i$.
Locally we can express $s$ and $s'$  in the basis $\{e_i\}$ as $s=f_i e_i$ and $s'=g_i e_i$ where $f_i$ and $g_i$ are meromorphic functions in $x$, the holomorphic coordinate centered at $p$.  Then $\nu_p(s)=\max_i \left( \deg_{\frac{1}{x}} f_i +\alpha_i \right)$, and similarly for $s'$.
Precisely because $\nu_p(s)-\nu_p(s') \notin \Z$,
the maximum is not attained at the same index.
Consequently, the leading order parts are linearly independent, and the orders add.
\end{proof}

\begin{defn}
A \emph{filtered $SL(K)$-bundle} on $(C,D)$ is a filtered rank $K$
bundle $\cE$ on $(C,D)$
with a global section $\omega \in \wedge^K \cE$, which gives
a trivialization of $\wedge^K \cE$ on $C - D$, and has
$\nu_p(\omega) = 0$ for each $p \in D$.
\end{defn}

\begin{defn}
A \emph{filtered $SL(K)$-Higgs bundle} on $(C, D)$ is a pair
$(\cE, \varphi)$ where $\cE$ is a filtered
$SL(K)$-bundle on $(C,D)$, and $\varphi$ is a traceless meromorphic
section of $\cE$, holomorphic on $C - D$.
\end{defn}

\begin{defn} \label{def:pdeg} If $\cL$ is a filtered line bundle over $(C,D)$,
its \ti{parabolic degree} $\pdeg \cL \in \R$ is defined as follows. For
any $p \notin D$ we define $\nu_p$ to be the ordinary pole order.
Then, fix any meromorphic section $s$ of $\cL$, and let
\begin{equation}\label{eq:pdeg}
  \pdeg \cL = -\sum_{p \in C} \nu_p(s).
\end{equation}
(The sum runs over all points of $C$, but it only receives nonzero
contributions from finitely many points. It is straightforward to check
that $\pdeg \cL$ is independent of the chosen $s$.)
If $\cE$ is a filtered rank $K$ vector bundle, then we define
$\pdeg \cE = \pdeg \wedge^K \cE$.
\end{defn}

\subsection{Filtered bundles over \texorpdfstring{$(\C\PP^1, \{\infty\})$}{(CP1,infty)}}

This paper mainly concerns filtered bundles over $(\C\PP^1, \{\infty\})$.
Thus we develop a few basic facts about these here.

\begin{defn} For any $\alpha \in \R$, let $\cO(\alpha)$ be
the filtered line bundle over $(\CP^1, \{\infty\})$ defined as follows:
the $\cO_{\CP^1}(*\{\infty\})$-module is just $\cO_{\CP^1}(*\{\infty\})$
itself, and the filtration is by pole order at $\infty$ shifted by $-\alpha$.
\end{defn}

We will frequently use some elementary facts about $\cO(\alpha)$:
\begin{itemize}
\item $\cO(\alpha)$ comes with a canonical trivialization
away from $z = \infty$, by a section $e$, corresponding to the
element $1 \in \cO_{\C\PP^1}(*\{\infty\})$. This section has
$\nu_\infty(e) = -\alpha$; up to scalar multiple, it is the unique
section with this property which is regular away from $z = \infty$.
\item The most general section of
$\cO(\alpha)$ which is regular away from $z = \infty$
is of the form $s = f(z) e$ for a polynomial $f$,
and has $\nu_\infty(s) = - \alpha + \deg f$.
\item $\pdeg \cO(\alpha) = \alpha$.
\item For $\alpha \in \Z$, the filtered line bundle $\cO(\alpha)$
is equivalent to the usual line bundle $\cO(\alpha)$, when the latter
is equipped with the filtration by pole order at $\infty$.
\end{itemize}

Now we can state the analogue of Grothendieck's lemma
for filtered bundles over $(\C\PP^1, \{\infty\})$:

\begin{lem} \label{lem:parabolic-grothendieck}
Suppose that $\cE$ is a filtered $SL(K)$-bundle
over $(\C\PP^1, \{\infty\})$. Then there is a decomposition of filtered bundles
\begin{equation}
  \cE = \bigoplus_{i=1}^K \cL_i
\end{equation}
where each $\cL_i \simeq \cO(\alpha_i)$ for some $\alpha_i$, and $\sum \alpha_i = 0$.
\end{lem}
\begin{proof} The proof is parallel to a standard proof of the ordinary
Grothendieck lemma, found e.g. in \cite{MR2093043}.

We induct on $K$. Let $\beta$ be the
minimum value attained by $\nu_\infty$ on a global section of $\cE$
(to see that a minimum does exist, note that Serre vanishing says
there are no global sections of $\cP_0 \cE \otimes \cO(-n)$
for large enough $n$).
Fix a global section $\psi$ with $\nu_\infty(\psi) = \beta$.
Then $\psi$ spans a filtered line subbundle $\cL \subset \cE$,
with $\pdeg \cL = - \beta$.
The filtered bundle $\cE$ is an extension,
\begin{equation}
  0 \to \cL \to \cE \to \cE / \cL \to 0,
\end{equation}
and by the inductive hypothesis $\cE / \cL = \bigoplus_{i=1}^{K-1} \cL_i$.
Our main problem is to show that the extension is split.
This works out just as in the case of ordinary vector bundles over $\CP^1$:
the extension class lies in $H^1(\cL')$ for $\cL'$ a filtered
line bundle over $(\CP^1, \infty)$ with $\pdeg \cL' \ge 0$,
and this cohomology group vanishes. In the rest of the proof
we spell this out longhand.

Choose local splittings $s_0, s_\infty: \cE / \cL \to \cE$
over patches $U_0, U_\infty$.
The difference $s_0 - s_\infty$ lifts to
a map $t: \cE / \cL \to \cL$ over $U_{0} \cap U_\infty$.
By adjusting the choice of $s_0$
and $s_\infty$ we can adjust $t \to t - \delta_0 - \delta_\infty$
where $\delta_0,\delta_\infty$ are maps $\cE / \cL \to \cL$ over $U_0$,
$U_\infty$ respectively. Using the inductive hypothesis it thus suffices
to show that every $t: \cL_i \to \cL$ over $U_0 \cap U_\infty$
can be realized as $\delta_0 + \delta_\infty$.
To see this, trivialize $\Hom(\cL_i, \cL)$ by a section $s$ away from
$z = \infty$; then $t = f(z) s$ for some meromorphic
$f(z)$ with singularities at $z = 0$ and $z = \infty$.
Expanding $f(z)$ in a Laurent series, the terms
of degree $\ge 0$ extend over $z = 0$, while the terms
of degree $\le -\nu_\infty(s)$ extend over $z = \infty$.
Since $\nu_\infty(s) = \pdeg \cL_i - \pdeg \cL \le 0$,
every term extends either over $z = 0$ or over $z = \infty$,
which gives the desired splitting.

Finally, the fact that $\omega \in \wedge^K \cE$ has $\nu_\infty(\omega) = 0$
shows that $\pdeg \wedge^K \cE = 0$, which implies $\sum_i \pdeg \cL_i = 0$
as desired.
\end{proof}

\begin{rem}
An analog of Lemma \ref{lem:parabolic-grothendieck}
is true for divisors containing two points,
and can be proven by a similar argument. In contrast,
if $D$ consists of three or more points,
then not all filtered vector bundles over $(\CP^1, D)$
are direct sums of filtered line bundles.
\end{rem}

\subsection{Good filtered Higgs bundles}\label{sec:goodfiltered}

Now we are ready to introduce the ``diagonalizability'' conditions
on the Higgs fields near the singularities.

\begin{defn}  A filtered $SL(K)$-Higgs bundle $(\mathcal{E}, \varphi)$ on $(C,D)$
 is \emph{unramifiedly good}
 if near each point $p \in D$ there is
 \begin{itemize}
 \item[$\bullet$] a local holomorphic coordinate $u$ centered at $p$,
  \item[$\bullet$] a local decomposition of filtered bundles
  \begin{equation} \label{eq:filtered-decomp-1}
     \cE = \bigoplus_{i=1}^{K} \cL_i
   \end{equation}
   where each $\cL_i$ is a filtered line bundle,
  \item[$\bullet$] a choice of singular type $\left(\mathfrak{a}_i \in \frac{1}{u}\C[\frac{1}{u}]\right)_{i=1}^r$,
 \end{itemize}
 such that
 \begin{itemize}
  \item[$\bullet$] $\varphi$ respects the decomposition \eqref{eq:filtered-decomp-1}
  (let $\varphi_i$ denote the restriction to $\cL_i$),
   \item[$\bullet$] $\varphi_i-\de \mathfrak{a}_i$ is logarithmic,
   i.e. $(\varphi_i - \de \mathfrak{a}_i) (\cP_\alpha \cL_i) \subset \cP_{\alpha+1} \cL_i \otimes K_C$.
  \end{itemize}
\end{defn}

We want to consider bundles which are not unramifiedly good, but
merely good, i.e. they become unramifiedly good only after pulling back
to a ramified cover:

\begin{defn}
 A filtered $SL(K)$-Higgs bundle $(\mathcal{E}, \varphi)$
 on $(C, D)$
 is called \emph{good} if near each point $p \in D$
 there is
 \begin{itemize}
 \item[$\bullet$] a local holomorphic coordinate $x$ on $U \ni p$, with $x(p) = 0$,
 \item[$\bullet$] a ramified covering
  \begin{align}
  \psi: \tU &\rightarrow U \subset C\\ \nonumber
  u &\mapsto u^m=x
 \end{align}
 \end{itemize}
  such that $\psi^*(\cE, \varphi)$ is unramifiedly good on $\tU$.
  Here $\psi^* \cE$ is equipped with its natural filtered structure
  \cite{MochizukiToda}, such that for
pulled-back sections we have\footnote{The extra factor of $m$ here is required for consistency, since the local coordinate $t$ on the cover
must have
$\nu_{\psi^{-1}(p)}(t) = 1$, while on the base we have $\nu_p(x) = 1$,
and $\psi^* x = u^m$.
}
\begin{equation} \label{eq:pullback-weight}
\nu_{\psi^{-1}(p)}(\psi^* s) = m\nu_p(s).
\end{equation}

\end{defn}

\subsection{Good filtered Higgs bundles on \texorpdfstring{$(\C\PP^1, \{\infty\})$}{CP1}} \label{sec:CKN}

Next we introduce the specific class of good filtered Higgs
bundles on $(\CP^1, \{\infty\})$ which we study.

\begin{defn} \label{defn:category}
Let $\cC_{K,N}$ be the category of good filtered $SL(K)$-Higgs bundles
$(\mathcal{E}, \varphi)$ over $(\CP^1, \{\infty\})$,
where:

\begin{itemize}
\item On a disc $U$ around $z = \infty$ we choose the coordinate $x = z^{-1}$
and the ramified covering $\psi: \tU \to U$ given by
\begin{equation}
x(\psi(u)) = u^{2K}.
\end{equation}
\item The singular type is
\begin{equation} \label{eq:singular-type}
  \fa_r = \frac{K}{K+N} \e^{2 \pi \I r/K} u^{-2(N+K)}.
\end{equation}
\item In the decomposition \eqref{eq:filtered-decomp-1} of $\psi^* \cE$ 
over $\tU$, each filtered line bundle $\cL_i$ is equivalent to an ordinary 
line bundle with the standard filtration by pole order at $u = 0$.
\end{itemize}
A morphism in $\cC_{K,N}$ is an
isomorphism of filtered bundles preserving the Higgs fields.
\end{defn}

Said otherwise, the Higgs bundles in $\cC_{K,N}$ are ones for which there exists a trivialization $g: \psi^* \cE \to \cO^{\oplus K}_\tU$ with
\begin{equation} \label{eq:gauge}
g \left( \psi^*\varphi \right) g^{-1} = \varphi_{\mathrm{model}} + (\text{holomorphic in } u),
\end{equation}
where
\begin{equation}\label{eq:localmodelphi}
\varphi_{\mathrm{model}}=-2K\begin{pmatrix} \e^{2 \pi \I / K} &  & & \\ & \e^{4 \pi \I / K} & & \\ & & \ddots & \\ & & & 1 \end{pmatrix}
 \frac{\de u}{u^{2K+2N+1}},
\end{equation}
and the filtration on $\psi^* \cE$ is induced from the standard filtration
on $\cO^{\oplus K}_{\tU}$.

\begin{defn}
Let $\cM_{K,N}$ be the set of objects in $\cC_{K,N}$
up to isomorphism.
\end{defn}

In this paper we will only treat $\cM_{K,N}$ as a set,
although we expect it to be a coarse moduli space, and to
carry many of the same structures
as the familiar moduli spaces $\cM_{K}(C)$, as
described in the introduction.

\subsection{The Hitchin base}

As with $\cM_{K}(C)$, we define the Hitchin map on $\cM_{K,N}$
by taking characteristic polynomials of Higgs fields:
\begin{equation}
	\pi([(\cE, \varphi)]) = \charp \varphi = \det (\lambda - \varphi(z)).
\end{equation}
Let $\cB_{K,N}$ denote the image of $\pi$. Then:
\begin{prop} \label{prop:hitchinbase}
$\cB_{K,N}$ is the space of polynomials of the form
\begin{equation} \label{eq:charpoly-form}
 (\lambda^K - z^{N} \de z^K) + (P_{2}(z) \de z^2 \lambda^{K-2} + \cdots + P_i(z) \de z^i \lambda^{K-i} + \cdots + P_K(z) \de z^K),
\end{equation}
where each $P_i(z)$ is a polynomial, with
\begin{equation}
 \deg P_i \le \frac{N(i-1)}{K} - 1.
\end{equation}
\end{prop}

For example,
\begin{itemize}
\item
$\cB_{2,N}$ is the space of polynomials $P_2$ with
\begin{equation} \label{eq:B-cond-2}
  \deg P_2 \le \frac{N}{2} - 1.
\end{equation}
\item
$\cB_{3,N}$ is the space of pairs $(P_2, P_3)$ with
\begin{equation} \label{eq:B-cond-3}
  \deg P_2 \le \frac{N}{3} - 1, \quad \deg P_3 \le \frac{2N}{3} - 1.
\end{equation}
\end{itemize}
In general $\mathcal{B}_{K,N}$ is an affine
space of complex dimension
\begin{equation}
\dim \cB_{K,N} = \sum_{i=2}^K \left\lfloor\frac{N(i-1)}{K} \right\rfloor = \frac{1}{2}(K-1)(N-1).
\end{equation}

\begin{proof}
Near $\infty$, the eigenvalues of $\varphi$ are 
\begin{equation}
\lambda_r = -2K\left(\e^{2 \pi \I r / K} u^{-(2K+2N+1)} + f_i(u)\right) \de u
\end{equation}
where $f_i(u)$ is holomorphic. Using $u^{-2K} = z$ this gives
\begin{equation} \label{eq:eigenvalues}
\charp \varphi
= \prod_{r=1}^K \left(\lambda - \e^{2 \pi \I r / K} z^{N/K} \de z - f_i(z^{-1/2K}) z^{-1-1/2K} \de z \right).
\end{equation}
Multiplying out gives
\begin{equation} \label{eq:degP}
	\deg P_i \le (i-1)\left(\frac{N}{K}\right) + \left(-1-\frac{1}{2K}\right).
\end{equation}
Since $N(i-1)$ is not a multiple of $K$ and $\deg P_i$ is necessarily
an integer, we can drop the last $\frac{1}{2K}$
to get
\begin{equation}
	\deg P_i \le (i-1)\left(\frac{N}{K}\right) - 1
\end{equation}
as desired.

To see that $\pi$ is surjective onto $\cB_{K,N}$,
we directly construct a family of filtered Higgs bundles
analogous to the Hitchin section \cite{hitchin87, hitchinteichmuller} in $\cM_{K}(C)$.
To define this family, we first pick a principal $\mathfrak{sl}(2,\C) \subset \mathfrak{sl}(K,\C)$:
\begin{equation}
X^+=\begin{pmatrix} 0 & 1 & & \\ & \ddots & \ddots & \\ & & 0 & 1 \\ & & & 0 \end{pmatrix}, \quad X^- = \begin{pmatrix} 0 &  & & \\ r_1 & 0 & & \\ & \ddots& \ddots &  \\ & & r_{K-1}& 0 \end{pmatrix}, \quad H=[X^+, X^-],
\end{equation}
with $r_i=i(K-i)$.
Let $X_1, \dots, X_{K-1} \in {\mathfrak{sl}}(K,\C)$
be the unique (up to scalar multiplication) matrices
such that $X_{i}$ has nonzero entries only on the $i$-th subdiagonal
and $[X_i, X^-]=0$.
Then, for polynomials $(Q_2(z), \dots, Q_K(z))$
with $\deg Q_i \le \lfloor \frac{N(i-1)}{K} -1 \rfloor$,
we consider
\begin{equation} \label{eq:Hitchinsec}
\cE = \bigoplus_{i=1}^K \, \cO \left(-\frac{N(K+1-2i)}{2K}\right), \quad
\varphi = \varphi_0 +
\sum_{i=2}^K X_{i-1} Q_i \de z, \quad \varphi_0 = \begin{pmatrix}0 & 1 & & \\  & \ddots & \ddots & \\ & & 0& 1 \\ z^N & & & 0 \end{pmatrix} \de z.
\end{equation}
To see that $(\cE, \varphi) \in \cC_{K,N}$, take 
   \begin{equation}\label{eq:gaugeHitchinsec}
 g = \hat{g} \cdot {\mathrm {diag}} (u^{-N(K+1-2i)})_{i=1}^K , \qquad
     \hat{g}_{jk} = \left( \e^{\frac{2 \pi \I}{K}} \right)^{jk}.
  \end{equation}
We want to show that
$g\left( \psi^*\varphi\right) g^{-1}- \tilde{\varphi}_{\mathrm{model}}$
is holomorphic in $u$. 
We compute:
\begin{equation}
	g (\psi^* \varphi_0) g^{-1} = \tilde{\varphi}_{\mathrm{model}},
\end{equation}
and
\begin{equation} \label{eq:pieceHiggs}
g \psi^*(X_{i} Q_{i+1} \de z)  g^{-1}
= \hat{g} \left( X_i Q_{i+1} z^{-\frac{Ni}{K}} \de z \right) \hat{g}^{-1}.
\end{equation}
Since
$\deg_z (Q_{i+1} z^{-\frac{Ni}{K}})\leq-1- \frac{1}{K}$
and $\de z \sim u^{-(2K+1)} \de u$,
\eqref{eq:pieceHiggs} is holomorphic in $u$,
as needed.

To see that the filtration on $\psi^*\cE$ is induced from the standard
filtration on $\cO^{\oplus K}$, recall that
$\cO(\alpha_i)$ has a canonical trivialization away from $z=\infty$,
by a section $e_i$ with $\nu_\infty(e_i)=-\alpha_i$.
Since $\psi$ is a $2K:1$ cover, $\nu_\infty( \psi^*e_i)=2K \nu_\infty(e_i)$.
From \eqref{eq:gaugeHitchinsec}, note that the gauge transformation $\hat{g}^{-1} g$ 
is diagonal and acts by multiplication by $u^{-N(K+1-2i)}$ on $e_i$,
hence $\nu_\infty(\hat{g}^{-1} g \psi^*e_i) = \nu_\infty(\psi^* e_i) -N(K+1-2i)$.
The gauge transformation $\hat{g}$ is regular at $z = \infty$, so
$\nu_\infty(g \psi^* e_i)=\nu_\infty(\hat{g}^{-1} g \psi^*e_i)$.
Altogether we get that $\nu_\infty(g \psi^* e_i)=0$, as claimed.

Finally, to see that this family maps surjectively onto $\cB_{K,N}$,
note that the coefficient of $\lambda^k$ in $\charp \varphi$ is $(-1)^k \mathrm{tr} (\wedge^k \varphi)$, where $\wedge^k \varphi: \wedge^k(\cE ) \rightarrow
\wedge^k( \cE  \otimes K_{\CP^1})$ is the induced
map.
Consequently, $P_k$ is related to $(Q_2, \dots, Q_K)$ by
\begin{equation} \label{eq:partitionP}
P_k = \sum_{i_1+\cdots + i_{m}=k} c_{i_1,\dots, i_m} \prod_{j=1}^m Q_{i_j}
 \end{equation}
 for some constants $c_\bullet$.  Given $P_2, \dots, P_K$,
 the corresponding $Q_k$ can be inductively determined: $Q_2$ can be determined from $P_2$, $Q_3$ can be determined
from $P_3$ and $Q_2$, etc.
To see that the $Q_k$ thus obtained have degree at most
$\lfloor \frac{N(k-1)}{K}-1 \rfloor$ as claimed, note that
$P_k$ has degree at most $\lfloor \frac{N(k-1)}{K}-1 \rfloor$
and all the other terms in \eqref{eq:partitionP} also
have degree (strictly) less than $\lfloor \frac{N(k-1)}{K}-1 \rfloor$.
\end{proof}

\begin{rem}\label{rem:irreducible}
One immediate consequence of this description of $\cB_{K,N}$ is that
every good filtered Higgs bundle $(\cE, \varphi) \in \cC_{K,N}$
is stable, for the simple reason that it has no proper
Higgs subbundle, since every polynomial in $\cB_{K,N}$
is irreducible over $\C[z]$. Indeed, \eqref{eq:eigenvalues}
gives its factorization over $\C[z^{1/2K}]$,
and the product of a proper subset of the factors --- say
$r$ of them, with $0 < r < K$ ---
cannot be in $\C[z]$, since the highest-degree part would
have degree $r \frac{N}{K}$ in $z$.
\end{rem}

\subsection{The stratification of \texorpdfstring{$\mathcal{M}_{K,N}$}{M(K,N)}}

The filtered bundles $\cE$ which can appear in pairs $(\cE, \varphi) \in \cC_{K,N}$
are of a special kind, as we now explain.
For convenience we introduce the notation
\begin{equation} \label{eq:tildenotation}
 \bar\varphi = \varphi / \de z.
\end{equation}

\begin{lem} \label{lem:nu-phi}
Suppose $(\mathcal{E}, \varphi) \in \mathcal{\cC}_{K,N}$.
For any section $s$ of $\cE$ in a
neighborhood $U$ of $\infty$, we have
\begin{equation} \label{eq:nu-shift}
  \nu_\infty(\bar\varphi(s)) = \nu_\infty(s) + \frac{N}{K}.
\end{equation}

\end{lem}
\begin{proof}
We compute ``upstairs'' on the ramified cover $\tU$, using
\eqref{eq:gauge}.
The explicit formula \eqref{eq:localmodelphi}, together
with the fact that the filtration on the pullback
is the standard one, shows that
$\varphi_{\mathrm{model}} / \de u$
raises the weight by $2K+2N+1$. Since $\de z / \de u \sim u^{-2K-1}$,
it follows that $\varphi_{\mathrm{model}} / \de z$
raises the weight by $2N$.
Recalling from \eqref{eq:pullback-weight}
that the downstairs weights differ from the
upstairs weights by a factor $2K$, we get the desired
\eqref{eq:nu-shift}.
\end{proof}

\begin{lem} \label{lem:nu-values}
Suppose $(\mathcal{E}, \varphi) \in \mathcal{\cC}_{K,N}$.
Then $\nu_\infty$ on sections of $\cE$ attains $K$ distinct values
mod $1$, differing by multiples of $\frac{1}{K}$.
\end{lem}
\begin{proof}
Consider any nonvanishing section $s$ of $\cE$ near $z = \infty$.
Using \eqref{eq:nu-shift}, we see that
$s, \bar \varphi s, \bar \varphi^2 s, \dots, \bar\varphi^{K-1} s$
all have different values of $\nu_\infty$ mod $1$, differing
by multiples of $\frac{1}{K}$.
\end{proof}

\begin{defn}\label{defn:cyclic}\hfill
\begin{itemize}
 \item An
\textbf{ordered $K$-partition of $N$} is a $K$-tuple
$\mathbf{b} = (b_1, \dots, b_K) \in \mathbb \Z_{\ge 0}^K$ with $\sum_{i=1}^K b_i = N$.
\item A
\textbf{cyclic $K$-partition of $N$} is an equivalence class $[\mathbf{b}]$
of ordered $K$-partitions of $N$, where
$\mathbf{b}$ and $\mathbf{b'}$ are equivalent if they differ by
a cyclic permutation of the index set $\{1, \dots, K\}$.
\end{itemize}
\end{defn}
For example,
\begin{itemize}
\item If $K=2$ and $N=3$, there are $2$ cyclic $K$-partitions of $N$:
$[(0,3)]$ and $[(1,2)]$.
\item If $K=3$ and $N=4$, there are $5$ cyclic $K$-partitions of $N$:
$[(1,1,2)]$, $[(0,1,3)]$, $[(0,3,1)]$, $[(0,2,2)]$, and $[(0,0,4)]$.
\end{itemize}

\begin{prop}\label{prop:shape}
Suppose $(\cE, \varphi) \in \mathcal{\cC}_{K,N}$.
There is a decomposition of filtered bundles
\begin{equation} \label{eq:filtered-decomp}
	\cE = \bigoplus_{i=1}^K \cL_i
\end{equation}
where each $\cL_i \simeq \cO(\alpha_i)$.
Moreover, if we define $\bfb$ by
\begin{equation}\label{eq:btoc-1}
 b_i - \frac{N}{K} = \alpha_{i}-\alpha_{i+1},
\end{equation}
then $[\mathbf{b}]$ is a cyclic $K$-partition of $N$,
canonically determined by the bundle $(\cE, \varphi)$.

\end{prop}

\begin{proof}
The existence of a decomposition of the form \eqref{eq:filtered-decomp}
follows from Lemma \ref{lem:parabolic-grothendieck}.
Moreover, using Lemma \ref{lem:nu-values},
we may assume, without loss of generality,
that the weights are ordered so that
\begin{equation}\label{eq:ordering}
-\frac{N}{K} = \alpha_{i} - \alpha_{i+1} \pmod  1.
\end{equation}
Each $\mathcal{O}(\alpha_j)$ comes with a canonical
trivialization away from $z=\infty$, by a section $e_j$
such that $\nu_\infty(e_j)=-\alpha_j$.
Expanding in this basis, $\bar\varphi(e_j)= \sum_{i=1}^K\bar{\varphi}_{ij} e_i$ for $\bar\varphi_{ij}$
holomorphic functions in $z$.
The $\nu_\infty(e_i)$ (and similarly, $\nu_\infty(\bar{\varphi}_{ij} e_i)$) are all distinct mod 1; hence, there is no cancellation and
\begin{equation}\label{eq:max}
 \nu_\infty(\bar\varphi(e_j)) = \max_{i}  \left( \nu_\infty(e_{i}) + \mathrm{deg}_z(\bar \varphi_{ij}) \right).
\end{equation}
The maximum must occur at an index
$i$ such that $\nu_\infty( \bar \varphi(e_j)) = \nu_\infty(e_i) \pmod 1$.
By Lemma \ref{lem:nu-phi},
$\nu_\infty(\bar\varphi(e_j)) = \nu_\infty(e_j) + \frac{N}{K}$; using this and
\eqref{eq:ordering} it follows that the maximum in \eqref{eq:max} is attained at the index $i=j-1$.
Define $b_{i}:=\mathrm{deg}_z(\bar \varphi_{i,i+1})$, and note that $b_i \ge 0$ since
$\bar \varphi_{i,i+1}$
is holomorphic.
The equation
\begin{equation}
 \nu_\infty(e_{i+1}) + \frac{N}{K} = \nu_\infty(\bar \varphi(e_{i+1}))=\nu_\infty(e_i) + \deg_z \bar \varphi_{i,i+1}
\end{equation}
gives
\begin{equation}\label{eq:btoc-reordered}
 -\alpha_{i+1} + \frac{N}{K} = -\alpha_{i} + b_i,
\end{equation}
proving \eqref{eq:btoc-1}.
Summing \eqref{eq:btoc-reordered} over $i$, we see that $\sum b_i=N$,
hence $\mathbf{b}$ is a $K$-partition of $N$.
The parabolic degrees $\alpha_i$ are determined up to cyclic
permutation by \eqref{eq:ordering}; consequently the cyclic
partition $[\mathbf{b}]$ is well-defined.
\end{proof}

Proposition \ref{prop:shape} gives a decomposition of $\cM_{K,N}$,
\begin{equation} \label{eq:stratification}
 	\cM_{K,N} = \bigsqcup_{[\bfb]} \cM_{K,N}^{[\bfb]},
 \end{equation}
analogous to the Bialynicki-Birula stratification for the usual
moduli spaces $\cM_K(C)$.
\ti{A priori} some of the strata could be empty, but we will
rule this out in Proposition \ref{prop:classification-fixed-points} below,
by explicitly exhibiting Higgs bundles in all strata.

\section{The \texorpdfstring{$\C^\times$}{C}-action and its fixed points}\label{sec:facts}

Recall from \cite{hitchin87} that on the moduli space $\cM_{K}(C)$ of $SL(K)$-Higgs bundles
without singularities on a compact curve $C$,
there is a $\C^\times$-action which rescales the
Higgs field:
\begin{equation} \label{eq:simple-cstar-action}
  (\cE, \varphi) \mapsto (\cE, \e^t \varphi), \qquad t \in \C / 2 \pi \I \Z.
\end{equation}
In this section we study an analogous $\C^\times$-action on $\cM_{K,N}$.

\subsection{The \texorpdfstring{$\C^\times$}{C}-action on \texorpdfstring{$\cM_{K,N}$}{M(K,N)}}

For $\cM_{K,N}$ the simple formula \eqref{eq:simple-cstar-action} will not work. One quick way to see that
it will not work is to note that the property
$\charp \varphi \in \cB_{K,N}$ is not preserved by
rescaling $\varphi$. Indeed, this condition is roughly
\begin{equation}
  \charp \varphi = \lambda^K - z^N \de z^K + \text{lower order terms},
\end{equation}
and this condition is not invariant under rescaling $\varphi \mapsto \e^t \varphi$: rather,
\begin{equation}
  \charp \e^t \varphi = \lambda^K - \e^{K t} z^N \de z^K + \text{lower order terms}.
\end{equation}
Luckily, there is a simple modification of the $\C^\times$-action which
\ti{does} preserve this condition: we need to combine $\varphi \mapsto \e^t \varphi$ with a compensating action on the base $\CP^1$,
\begin{equation}
  \rho_t(z) = \e^{-\frac{K t}{K+N}} z.
\end{equation}
Thus instead of \eqref{eq:simple-cstar-action} we consider:
\begin{equation} \label{eq:new-cstar-action}
  (\cE, \varphi) \mapsto (\rho_t^* \cE, \e^t \rho_t^* \varphi), \qquad t \in \C / 2 \pi \I (K+N) \Z.
\end{equation}

\begin{figure}[h]
\begin{centering}
\includegraphics[width=1.2in]{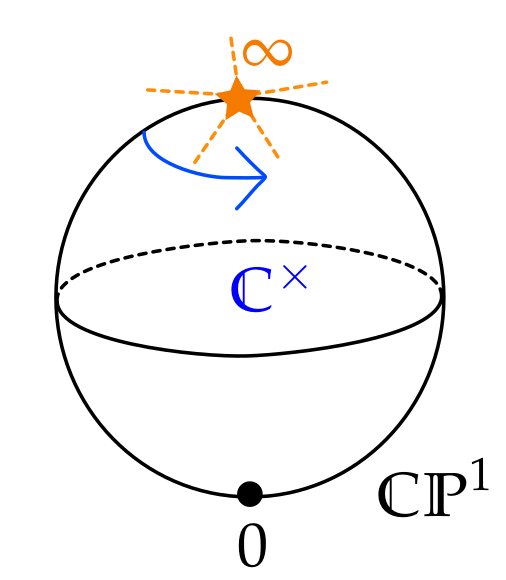}
\caption{\label{fig:CP1}
The $\C^\times$-action on $\CP^1$ fixes the marked point at $z = \infty$
as well as the point $z = 0$.
We indicate $K+N$ ``Stokes rays'' around $z = \infty$,
determined by the singularity data $\fa_i$ there, given
in \eqref{eq:singular-type}. The singularity data are
not preserved by the rotation of $\CP^1$ alone, but are
preserved by the combined action \eqref{eq:new-cstar-action}.}
\end{centering}
\end{figure}

\begin{prop} The action \eqref{eq:new-cstar-action} maps $\cC_{K,N} \to \cC_{K,N}$.
\end{prop}

\begin{proof}
The action of $\rho_t$ naturally lifts
to the ramified cover $\tU$, where it maps
\begin{equation}
  \rho_t(u) = \e^{\frac{t}{2(K+N)}} u,
\end{equation}
and this lifted action obeys
\begin{equation} \label{eq:lifted-model-action}
 \e^t \rho_t^* \tilde{\varphi}_\mathrm{model} = \varphi_\mathrm{model}.
\end{equation}
Suppose $(\cE, \varphi) \in \cC_{K,N}$; then for some $g$,
 \begin{equation} \label{eq:lift1}
  g\left( \psi^* \varphi \right) g^{-1} - \varphi_{\mathrm{model}}
 \end{equation}
 is holomorphic in $u$. To see that $(\rho_t^* \cE, \e^t \rho_t^* \varphi) \in \cC_{K,N}$, just
 apply $\rho_t^*$ to everything in \eqref{eq:lift1} to conclude that
 \begin{equation} \label{eq:lift2}
g_t \left(\psi_t^* (\e^t \rho_t^*\varphi) \right) g_t^{-1} - \e^t \rho_t^* \varphi_{\mathrm{model}}
 \end{equation}
 is also holomorphic in $u$. Then using \eqref{eq:lifted-model-action} we get
 $(\rho_t^* \cE, \e^t \rho_t^* \varphi) \in \cC_{K,N}$ as desired.
\end{proof}

\subsection{The central fiber}

The action \eqref{eq:new-cstar-action} descends to the moduli
space $\cM_{K,N}$. Moreover, it projects to an action on $\cB_{K,N}$,
which transforms
\begin{equation}
  P_r(z) \mapsto \e^{\frac{rN}{K+N} t} P_r(\e^{- \frac{K}{K+N} t} z).
\end{equation}
The only fixed point is $P_2 = \cdots = P_K = 0$.
Consequently, all fixed points of the $\C^\times$-action on $\cM_{K,N}$ lie in the fiber
over this point. This fiber in $\cM_{K,N}$ plays a role analogous to that of the global
nilpotent cone in $\cM_{K}(C)$; we call it the \ti{central fiber}.

\subsection{Fixed points of the \texorpdfstring{$\C^\times$}{C}-action} \label{sec:fixedpoints}

One of the key technical devices in the study of
$\cM_{K}(C)$ is an
analysis of the fixed locus of the $\C^\times$-action,
$F \subset \cM_{K}(C)$. In general $F$ is rather
complicated: it has various components of various dimensions.
One of the major obstacles to understanding the topology of $\cM_{K}(C)$
is the complicated nature of $F$. (For $K=2$ it
was already described in \cite{hitchin87}, for $K=3$ in \cite{Gothen94}, for $K=4$
in \cite{GPHS}. For $K>4$, $F$ is so complicated that other techniques
are better \cite{GPGM05}.)

For the $\cM_{K,N}$ the situation is
much simpler, as
we now show: the fixed locus $F \subset \cM_{K,N}$ is a finite set,
and more precisely,
there is a single fixed point in each of the strata
$\cM_{K,N}^{[\bfb]}$ of \eqref{eq:stratification}.

\begin{prop} Let $\bfb$ be an ordered $K$-partition of $N$, and let
\begin{equation} \label{eq:b-higgs-bundle}
  \cE_\bfb = \bigoplus_{i=1}^K \cO(\alpha_i), \qquad \varphi_\bfb = \begin{pmatrix}
               0 & z^{b_1} & &\\ & &  \ddots & \\
                & & & z^{b_{K-1}}\\
                z^{b_K} & & &
              \end{pmatrix} \de z,
\end{equation}
where the $\alpha_i$ and $b_i$ are related by
\begin{equation}\label{eq:btoc}
 b_i - \frac{N}{K} = \alpha_{i}-\alpha_{i+1}, \qquad \sum_{i=1}^K \alpha_i = 0.
\end{equation}
Then $(\cE_\bfb, \varphi_\bfb) \in \cC_{K,N}$. Moreover, if we define
$g_t: \cE_{\bfb} \to \rho_t^* \cE_{\bfb}$ by
\begin{equation} \label{eq:gzeta-b-fixed}
  g_t(e_i) = \e^{- t \frac{K}{K+N} \alpha_i} \rho_t^* e_i,
\end{equation}
then we have
\begin{equation} \label{eq:b-fixed}
  \e^t \rho_t^* \varphi_\bfb = g_t \varphi_\bfb g_t^{-1}.
\end{equation}
In particular, $[(\cE_\bfb, \varphi_\bfb)] \in \cM_{K,N}$ is a $\C^\times$-fixed point,
depending only on $[\bfb]$.
\end{prop}

\begin{proof} First we check that
$(\cE_\bfb, \varphi_\bfb) \in \cC_{K,N}$.
Define $g: \cO^{\oplus K} \to \psi^* \cE_\bfb$ by
\begin{equation}
	g = \hat{g} ({\mathrm {diag}} (u^{2K \alpha_i})_{i=1}^K) , \qquad
     \hat{g}_{jk} = \left( \e^{\frac{2 \pi \I}{K}} \right)^{jk}.
\end{equation}
Then we have on the nose
\begin{equation}
	g \psi^* \varphi_\bfb g^{-1} = \phimodel.
\end{equation}
Moreover, using the trivial filtration on $\cO^{\oplus K}$ we have
$\nu_\infty(g \psi^* e_i) = -2K \alpha_i$,
while using the filtration on $\cE$ we have
$\nu_\infty(e_i) = - \alpha_i$. Thus the pullback filtration
on $\psi^* \cE$ indeed matches with the standard filtration
on $\cO^{\oplus K}$. This shows that $(\cE_\bfb,
\varphi_\bfb) \in \cC_{K,N}$.

Then we check \eqref{eq:b-fixed} by direct calculation:
in $\e^t \rho_t^* \varphi_\bfb$ the monomial $z^{b_i}$
is multiplied by the factor $\e^{t \left(1-\frac{K(b_i+1)}{K+N}\right)}$, and in $g_t \varphi_\bfb g_t^{-1}$
it is multiplied by
$\e^{-t (\alpha_i - \alpha_{i+1}) \frac{K}{K+N}}$. These two match, using \eqref{eq:btoc}.
\end{proof}

\begin{rem}
The conditions \eqref{eq:btoc} are equivalent to
\begin{equation} \label{eq:btocviaB}
\boldsymbol{\alpha}= -B \mathbf{b},
\end{equation}
where the entries of $B$ are
\begin{equation}\label{eq:B}
 B_{ij}=\frac{1}{2K} \left( -(K-1)+2 \left( \;(j-i)(\mathrm{mod}\; K)\; \right) \right).
\end{equation}
For example,
\begin{equation}
  B^{K=2} = \frac14 \begin{pmatrix} -1 & 1 \\ 1 & -1 \end{pmatrix}, \qquad   B^{K=3} = \frac16 \begin{pmatrix} -2 & 0 & 2 \\ 2 & -2 & 0 \\ 0 & 2 & -2 \end{pmatrix}.
\end{equation}

\end{rem}

\begin{rem}
There is an involution $r$ on $K$-partitions of $N$
defined by $r((b_1, \dots, b_K)) = (b_K, \dots, b_1)$.
This involution corresponds to taking duals:
\begin{equation}
	(\cE_{r(\bfb)}, \varphi_{r(\bfb)}) = (\cE_\bfb^*, \varphi_\bfb^T).
\end{equation}

\end{rem}

So far we have found a collection of $\C^\times$-fixed points
in $\cM_{K,N}$
labeled by cyclic $K$-partitions of $N$. Next we show that
these are all the fixed points:

\begin{prop} \label{prop:classification-fixed-points} The map
\begin{equation}
  [\mathbf{b}] \mapsto [(\cE_\bfb, \varphi_\bfb)]
\end{equation}
gives a bijection
\begin{equation}
 \left\{\mbox{cyclic $K$-partitions of $N$} \right\} \rightarrow \left\{ \C^\times\text{-fixed points in } \cM_{K,N} \right\}.
\end{equation}
\end{prop}

\begin{proof}
Suppose $[(\cE, \varphi)]\in \cM_{K,N}$
is a $\C^\times$-fixed point.
What we need to show is that $(\cE, \varphi)$
is equivalent to some $(\cE_\bfb, \varphi_\bfb)$,
with $\bfb$ unique up to cyclic permutation.

The fixed-point property means that for each $t \in \C$
there is an isomorphism of filtered $SL(K)$-bundles
\begin{equation}
 g_t: \cE \rightarrow \rho_t^* \cE
\end{equation}
such that
\begin{equation} \label{eq:gzeta-higgs}
\e^t \rho_t^* \varphi = g_{t} \varphi g_{t}^{-1}.
\end{equation}
The first key fact we need is that $g_t$ is diagonalizable,
in the sense that there exists a decomposition of $\cE$ into
filtered line bundles $\cL_i$ such that $g_t$ maps
$\cL_i \to \rho_t ^* \cL_i$.

This is slightly trickier than it sounds: we do have the decomposition
 \eqref{eq:filtered-decomp}
of $\cE$ provided by Proposition \ref{prop:shape}, but that decomposition
is not unique.
Thus we must consider the space $L(\cE)$
of all filtered line decompositions of
$\cE$. $L(\cE)$
is a flag manifold $G / S$, where $G$ is the group of endomorphisms of $\cE$,
and $S$ is the
stabilizer of the decomposition \eqref{eq:filtered-decomp}.
$G$ is an upper-triangular group, since there
are no holomorphic maps $\cL_i \to \cL_j$ when $\pdeg \cL_i > \pdeg \cL_j$,
and $S$ is its diagonal subgroup. It follows
that $L(\cE) \simeq \bigoplus_{i \neq j} \Hom(\cL_i, \cL_j)$,
and in particular $L(\cE)$ is an affine space.

For any $t$ we also have $L(\cE) = L(\rho_t^* \cE)$
canonically.
Thus the operators $g_t$ act on $L(\cE)$.
Moreover we have $g_t g_{t'} = g_{t + t'}$
acting on $L(\cE)$ (to see this, note that $g_{t+t'}^{-1} \circ
\rho_t^* g_t' \circ g_t$ is an automorphism of the Higgs bundle
$(\cE,\varphi)$, but this Higgs bundle is irreducible, so its only automorphisms
are scalar multiplications by $K$-th roots of unity.)
Thus the $g_t$ induce an action of $\C^\times$ on the affine space
$L(\cE)$. Such an action necessarily has a fixed point since $L(\cE)$ has
Euler characteristic $1$.
This gives the desired line decomposition of $\cE$.

After modifying the $g_t$ by $K$-th roots of unity we can arrange
$g_t g_{t'} = g_{t+t'}$ on the nose, and $g_{2 \pi \I (K+N)}$ is an
automorphism of $(\cE, \varphi)$, thus a $K$-th root of unity.
We trivialize each $\cL_i$ by a section $e_i$ over $\CP^1 - \infty$,
and trivialize $\rho^* \cL_i$ by $\rho^* e_i$.
Then $g_t$ is represented by a diagonal matrix:
\begin{equation}
  g_t = \diag((\e^{- t \frac{1}{K+N} \gamma_i})_{i=1}^K),
\end{equation}
for some $\gamma_i \in \frac{1}{K} \Z$.

Using \eqref{eq:gzeta-higgs},
$\e^t \rho^*_t \varphi_{ij} = (g_t)_{ii} (g_t)^{-1}_{jj}
\varphi_{ij}$, which determines $\varphi_{ij}$ to be a monomial in $z$:
\begin{equation}
\varphi_{ij} = c_{ij} z^{\beta_{ij}} \de z, \qquad \beta_{ij} = \frac{\gamma_i - \gamma_j + N}{K}.
\end{equation}
Thus $\beta_{ij} \in \Z$ only
if $\gamma_j - \gamma_i = N \pmod K$; for other $(i,j)$ we must have
$\varphi_{ij} = 0$. Since $\det \varphi \neq 0$ there must exist
at least one permutation $\sigma \in S_K$ for which $\gamma_{\sigma(i)} - \gamma_i = N \pmod K$. Since $(N,K) = 1$ there is at most one such
permutation. Thus, by reordering the $\cL_i$ we may arrange that
$\gamma_{i+1} - \gamma_i = N \pmod K$, and this ordering is unique up
to a cyclic permutation. It follows that the only nonzero entries of $\varphi$
are the $\varphi_{i,i+1}$.
Let $b_i = \beta_{i,i+1}$.
By scalar rescalings of the
$e_i$ we may arrange that $\varphi_{i,i+1} = z^{b_i} \de z$.

Next we prove the relation \eqref{eq:btoc}.
For this we use the fact in Lemma \ref{lem:nu-phi},
\begin{equation}
  \nu_\infty(\bar\varphi (e_{i+1})) = \nu_\infty(e_{i+1}) + \frac{N}{K},
\end{equation}
and
\begin{equation}
  \bar\varphi (e_{i+1}) = z^{b_i} e_i.
\end{equation}
Combining these we have
\begin{equation}
  b_i - \alpha_{i} = - \alpha_{i+1} + \frac{N}{K}
\end{equation}
which gives the desired \eqref{eq:btoc}.

Finally, we need to show that
${\mathbf b} = (b_1, \dots, b_K)$ is an ordered $K$-partition
of $N$.
We have $b_i \geq 0$ because $\varphi$ has no singularity at $z = 0$.
\eqref{eq:btoc} implies easily that $\sum_{i=1}^K b_i = N$.
\end{proof}

\subsection{Harmonic bundles at fixed points} \label{sec:harmonicbundles}

In this section we consider the \ti{harmonic metrics} associated to the $\C^\times$-fixed points
in $\cM_{K,N}$. We will see that they can be described in an explicit fashion.
The key fact which makes this possible is that for these Higgs bundles
the Hitchin
equation \eqref{eq:hitchin} reduces to an ODE in the radial coordinate.

We begin with some preliminaries about harmonic metrics on Higgs bundles
with singularities \cite{simpsonnoncompact}.

\begin{defn}
A hermitian metric $h$ on a filtered bundle $\cE$ over $(C,D)$
is \emph{adapted to the filtration}
if, in a local holomorphic coordinate $x$ centered at $p \in D$,
\begin{equation}
\cP_\alpha \cE = \left\{ s:  \|s\|_h =O\left(|x|^{-(\alpha + \eps)} \right) \mbox{for all $\eps>0$} \right\}
\end{equation}
for all $\alpha \in \R$.
\end{defn}

\begin{defn}
Given a good filtered Higgs bundle $(\cE, \varphi)$ over $(C,D)$
a \ti{harmonic metric} on $(\cE, \varphi)$ is a Hermitian metric on $\cE$,
adapted to the filtration,
such that
\begin{equation} \label{eq:hitchin}
 F_{D(\cE, h)} + \left[\varphi, \varphi^{\dagger_h}\right]=0,
\end{equation}
where $D(\cE, h)$ is the Chern connection
and $\varphi^{\dagger_h}$ is the hermitian adjoint of $\varphi$
with respect to $h$.
If $h$ is a harmonic metric we say $(\cE, \varphi, h)$ is a
\ti{harmonic bundle}.
\end{defn}

The key analytic fact is the following existence theorem:

\begin{thm} \label{thm:nab-hodge}
Given a good filtered Higgs bundle $(\cE, \varphi)$
over $(C,D)$
there is a harmonic metric $h$ on $(\cE, \varphi)$,
unique up to scalar multiple.
\end{thm}
Theorem \ref{thm:nab-hodge} is essentially proven in \cite{biquardboalch},
though one detail is missing: strictly speaking, that reference
treats only the \ti{unramifiedly good} case rather than simply
\ti{good}. The general theorem is stated as
Theorem 2.7 of \cite{MochizukiToda}.

Now we would like to understand what additional structure we get
if we consider Higgs bundles corresponding to $\C^\times$-fixed points
in $\cM_{K,N}$.
The first step is to show that the construction of harmonic metrics is
covariant for the action of $S^1 \subset \C^\times$:

\begin{prop} \label{prop:harmonic-s1-covariant}
Suppose $(\cE, \varphi, h)$ is a harmonic bundle and $t \in \I \R / 2 \pi \I (K+N)\Z$.
Then
$(\rho_{t}^* \cE, \e^t \rho_t^* \varphi, \rho_t^* h)$ is also a harmonic bundle.
\end{prop}

\begin{proof} This is a straightforward computation: the point is that the equation
\eqref{eq:hitchin} is invariant under this action, because the
term $[\varphi, \varphi^{\dagger_h}]$ is rescaled by $\e^{t} \e^{\bar{t}} = 1$.
(Note that this would not have been true for more general $t \in \C / 2 \pi \I (K+N)\Z$.)
\end{proof}

In particular, if $[(\cE, \varphi)]$ is fixed under $\C^\times$, then it is fixed
under $S^1 \subset \C^\times$, and then Proposition \ref{prop:harmonic-s1-covariant}
imposes a constraint on $h$:

\begin{prop}\label{prop:harmonic}
 Let $\mathbf{b}$ be an ordered $K$-partition of $N$.
 The harmonic metric $h_\bfb$ on the  good filtered
 Higgs bundle $(\cE_\bfb, \varphi_\bfb)$ of \eqref{eq:b-higgs-bundle} is
 \begin{align}
  h_\bfb &= \begin{pmatrix}\label{eq:hexpression}
     |z|^{-2 \alpha_1} \e^{u_1} & & \\
      & \ddots & \\
      & & |z|^{-2 \alpha_K} \e^{u_K}
    \end{pmatrix},
 \end{align}
where $u_i(z)=u_i(|z|)$, $\alpha_i$ are related to $b_i$ by \eqref{eq:btoc},
and the functions $u_i:\R^{>0} \rightarrow \R$ solve
 \begin{equation}\label{eq:ODE}
\frac{1}{4} \left(\frac{\de^2}{\de |z|^2}+ \frac{1}{|z|} \frac{\de}{\de|z|} \right) u_i =
|z|^{\frac{2N}{K}} \left(\e^{u_i-u_{i+1}} - \e^{u_{i-1}-u_i} \right)
\end{equation}
with the following boundary conditions:
\begin{itemize}
\item The function $u_i$ decays to $0$  as $\abs{z} \rightarrow \infty$.
\item Near $0$, $u_i \sim 2 \alpha_i\log{\abs{z}}$.
\end{itemize}
Additionally,
\begin{align}\label{eq:bc}
 \lim_{|z| \rightarrow 0} |z|^{\frac{2(K+N)}{K}} (\e^{u_i-u_{i+1}} -1) &=0, &
\underset{|z| \rightarrow 0}{\lim} |z| \frac{\de u_i}{\de|z|} &=2 \alpha_i, \\
 \lim_{|z| \rightarrow \infty} |z|^{\frac{2(K+N)}{K}} \sum_{i=1}^K(\e^{u_i-u_{i+1}} -1) &=0, &
 \underset{|z| \rightarrow \infty}{\lim} |z| \frac{\de u_i}{\de |z|} &=0. \label{eq:bc-2}
\end{align}
\end{prop}

\begin{proof}
Recall that the good filtered Higgs bundle $(\cE_\bfb, \varphi_\bfb)$ is fixed
up to equivalence by the $\C^\times$-action, with the specific equivalence
given in \eqref{eq:b-fixed}, \eqref{eq:gzeta-b-fixed}. It follows that for any
$t \in \C$
the harmonic metric on $(\rho_t^* \cE_\bfb, \e^t \rho_t^* \varphi_\bfb)$
is $(g_{t}^{-1})^{\dagger}  h g_{t}^{-1}$.
On the other hand, for $t \in \I \R$ we know from Proposition \ref{prop:harmonic-s1-covariant}
that the harmonic metric on $(\rho_t^* \cE_\bfb, \e^t \rho_t^* \varphi_\bfb)$
is $\rho_t^* h$. Combining these, for $t \in \I \R$ we obtain
\begin{equation} \label{eq:hentries}
\rho_{t}^* h= (g_{t}^{-1})^{\dagger} h g_{t}^{-1}.
\end{equation}
This has several consequences:
\begin{itemize}
\item Taking $t = 2 \pi \I \frac{K+N}{K}$ we get $\rho_t = 1$, so that \eqref{eq:hentries}
becomes a pointwise constraint on $h$,
\begin{equation}
  h_{ij} = \e^{2 \pi \I (\alpha_i - \alpha_j)} h_{ij}.
\end{equation}
Thus $h_{ij} = 0$ if $\alpha_i - \alpha_j \notin \Z$.
From \eqref{eq:btoc},
$\alpha_i - \alpha_j \in \Z$
only if $i=j$. Consequently $h$ is diagonal.
\item For the diagonal components of $h$, \eqref{eq:hentries} reduces to $\rho_t^* h_{ii} = h_{ii}$, which means $h$ is rotationally symmetric:
 $h_{ii}(z)=h_{ii}(|z|)$.
\end{itemize}

We define $u_i(|z|)$ by $h_{ii}(|z|)=|z|^{-2\alpha_i} \e^{u_i(|z|)}$.
Then the $(i,i)$-entry of \eqref{eq:hitchin}
gives the desired equation \eqref{eq:ODE}.

Next we consider the boundary behavior of $u$.
The behavior $u_i(|z|) \sim 2 \alpha_i \log \abs{z}$ near $0$ follows from the smoothness
of the entries of $h$.
Additionally, the limits at $\abs{z} = 0$ in \eqref{eq:bc}
follow.  Because $h$ is smooth across $\abs{z}=0$, $\frac{\de}{\de z} \left( \abs{z}^{-2\alpha_i} \e^{u_i} \right)=0$.
The rightmost limit in \eqref{eq:bc} holds, i.e. $\lim_{\abs{z}=0} \frac{\de u_i}{\de \abs{z}}=2 \alpha_i$.
The leftmost limit in \eqref{eq:bc} essentially comes from plugging
in $u_i \sim 2\alpha_i \log \abs{z}$ and noting that $\alpha_{i+1}-\alpha_i \leq \frac{K+N}{N}$.
This inequality is not sharp; from \eqref{eq:btoc}, the $\alpha_i$ satisfy $\alpha_{i+1}-\alpha_{i} = \frac{N}{K}-b_i$.
The functions $u_i(|z|)$ are bounded as $\abs{z} \to \infty$
because $h$ is adapted to the filtration of $\cE_\bfb$.
The stronger statement that $\lim_{|z| \to \infty} u_i(|z|)=0$,
and the properties in \eqref{eq:bc} at $\infty$,
follow from Lemma \ref{lem:u2} below,
which uses the maximum principle to give stronger
bounds on the norm of $\mathbf{u}(|z|)=(u_1, \dots, u_K)$.
\end{proof}

\begin{rem}\label{rem:changeofvar} By
the change of variables $\rho^2=\frac{2K}{K+N}|z|^{\frac{2(K+N)}{K}}$,
\eqref{eq:ODE} becomes the coupled system of ODE
\begin{equation}\label{eq:toda}
\left(\frac{\de^2}{\de \rho^2}+ \frac{1}{\rho} \frac{\de}{\de\rho} \right) u_i =
\e^{u_i-u_{i+1}} - \e^{u_{i-1}-u_i}.
\end{equation}
This is the radial
version of the coupled system of PDE known as ``2d cyclic affine Toda lattice with opposite sign.''
\end{rem}

\begin{rem} 
Essentially the same harmonic bundles appearing in Proposition \ref{prop:harmonic},
and the corresponding Toda lattice equations,
are also considered by Mochizuki in \cite{MochizukiToda}.
Mochizuki arrives at them in a different way:
rather than considering the set $\cM_{K,N}$,
his starting point is the specific Higgs field
\begin{equation}
 \varphi_{Moc} = \begin{pmatrix}
               0 & z^{-1} & &\\ & &  \ddots & \\
                & & & z^{-1}\\
                z^{K+N-1} & & &
              \end{pmatrix}\de z,
              \end{equation}
on $\C\PP^1$ with marked point at $z=\infty$ (irregular singularity)
and marked point at $z=0$ (regular singularity).
Proposition 3.17 of \cite{MochizukiToda} then gives a bijection
between the (continuous) set of compatible filtered bundle
structures and the set of parabolic weights at $0$.
The $\C^\times$-fixed points we have found in $\cM_{K,N}$ are not
equivalent to one another and have no singularity at $z=0$, but if
we nevertheless allow gauge transformations which are singular at $z=0$,
then all these fixed points become equivalent to $\varphi_{Moc}$.
In the language of \cite{MochizukiToda}
they correspond to some specific
choices of the parabolic weights at $z=0$.
\end{rem}

\begin{lem} \label{lem:u2}
Let $\mathbf{u}(\rho) = \left(u_1(\rho), \dots, u_K(\rho)\right)$ where
$u_i(\rho)$ solve the system \eqref{eq:toda} on $(0, \infty)$.
Then, $\mathbf{u}$ satisfies
\begin{equation}\label{eq:85}
  \frac{1}{2}
\left( \frac{\de^2}{\de \rho^2} + \frac{1}{\rho} \frac{\de}{\de \rho} \right)
\|\mathbf{u}\|^2 = \sum_{i} \left(u_i - u_{i+1} \right) \left(\e^{u_i-u_{i+1}} - 1 \right) +
\left\|\frac{\de\mathbf{u}}{\de \rho}\right\|^2.
\end{equation}
Suppose $u_i(\rho)$ is bounded near $\infty$.
Then $\|\mathbf{u}(\rho)\|^2$ is decreasing and
$\lim_{\rho \to \infty} \|\mathbf{u}\|^2=0$.
Moreover, $\|\mathbf{u}(\rho)\|^2$ exhibits exponential decay at $\infty$.
More precisely, for $\eps>0$, take $R_\eps>0$ such that $\|\mathbf{u}(R_\eps)\| < \eps$.
Then, there is a constant $c>0$ (depending explicitly on $\eps$ and $K$) such that
\begin{equation}\label{eq:85b}
\|\mathbf{u}\|^2 (\rho)  \leq \eps^2 \frac{ K_0(c \rho)}{K_0(c R_\eps)} \qquad \mbox{for $\rho>R_\eps$},
\end{equation}
where $K_0$ is the modified Bessel function of first kind.
\end{lem}
\begin{proof}
\eqref{eq:85} follows from manipulating \eqref{eq:toda}:
\begin{align}\label{eq:96}
\frac{1}{2}\left( \frac{\de^2}{\de \rho^2} + \frac{1}{\rho} \frac{\de}{\de \rho} \right)
\|\mathbf{u}\|^2 &= \sum_i u_i \left( \frac{\de^2}{\de \rho^2} + \frac{1}{\rho} \frac{\de}{\de \rho} \right) u_i
+ \left\|\frac{\de\mathbf{u}}{\de \rho}\right\|^2\\ \nonumber
&= \sum_{i} u_i \left( \e^{u_i-u_{i+1}} - \e^{u_{i-1} - u_i} \right) +
\left\|\frac{\de\mathbf{u}}{\de \rho}\right\|^2\\ \nonumber
&= \sum_{i} \left(u_i -u_{i+1} \right) \left(\e^{u_i-u_{i+1}} -1 \right) +
\left\|\frac{\de\mathbf{u}}{\de \rho}\right\|^2.
\end{align}

The function $\|\mathbf{u}\|^2$ has no maximum in $(0, \infty)$
unless $\|\mathbf{u}\|^2 \equiv 0$,
by the maximum principle. To see this, note that if $\|\mathbf{u}\|^2$ attained
a maximum at $\rho_0 \in (0, \infty)$, then the LHS of \eqref{eq:85}
would be non-positive while the RHS of \eqref{eq:85}
would be non-negative (since the function $f(x)=x(\e^x-1) \geq 0$ with equality at $x=0$).  Consequently, both sides would
vanish and $u_i(\rho_0)=0$, $\frac{\de u_i}{\de\rho}(\rho_0)=0$.
By the uniqueness of solution of the initial value problem (here, it's important that $\rho_0 \neq 0$),
all $u_i$ would vanish on $(0, \infty)$.

Now, suppose $\|\mathbf{u}\|^2$ is bounded at $\infty$.
Because $\|\mathbf{u}\|$ is bounded, non-negative with no maximum on $(0, \infty)$, it follows that
\begin{equation}
\lim_{\rho \to \infty} \left( \frac{\de^2}{\de \rho^2} + \frac{1}{\rho} \frac{\de}{\de \rho} \right)
\|\mathbf{u}\|^2=0.
\end{equation}
Consequently, the RHS of \eqref{eq:85} also converges to zero at $\infty$, and in particular
$\lim_{\rho \to \infty} u_i -u_{i+1} =0$.
Hence, $\lim_{\rho \to \infty} \|\mathbf{u}\|^2=0$. (Note that
as a consequence, $\|\mathbf{u}\|^2$ is decreasing.)

Given $\eps>0$, take $R_\eps>0$ such that $\|\mathbf{u}\|^2(R_\eps)<\eps^2$.
Then for $\rho>R_\eps$,
\begin{equation}\label{eq:85c}
\sum_i (u_i-u_{i+1})(\e^{u_i-u_{i+1}} -1) \geq \sum_i C_\eps (u_i-u_{i+1})^2
\end{equation}
for $C_\eps=\frac{1-\e^{-2\eps}}{2\eps} $.  To explain this choice of $C_\eps$,
take $x=u_i(\rho)-u_{i+1}(\rho)$, and note $|x|=|u_i-u_{i+1}| \leq 2 \|\mathbf{u}\| \leq 2\eps$
for $\rho>R_\eps$.
Then observe that $C_\eps= \sup_{x \in [-2\eps, 2\eps]} \frac{x(\e^x-1)}{x^2}.$

Consequently, by combining \eqref{eq:85} and \eqref{eq:85c}, note that
\begin{align}\label{eq:96b}
\frac{1}{2}\left( \frac{\de^2}{\de \rho^2} + \frac{1}{\rho} \frac{\de}{\de \rho} \right)
\|\mathbf{u}\|^2 &=  \sum_{i} \left(u_i -u_{i+1} \right) \left(\e^{u_i-u_{i+1}} -1 \right) +
\left\|\frac{\de\mathbf{u}}{\de \rho}\right\|^2\\ \nonumber
&\geq  C_\eps \sum_{i} \left(u_i -u_{i+1} \right)^2 \\ \nonumber
&= C_\eps \sum_{i} u_i \left(-u_{i-1} + 2 u_i -u_{i+1} \right) \\ \nonumber
&= C_\eps \mathbf{u}^T  (Q^TQ) \mathbf{u}
\end{align}
where the matrix $Q$ is
 \begin{equation}
Q_{ij}=\begin{cases}
1 \qquad \mbox{if}\; j=i\\
 -1 \quad \mbox{if}\; j-i=1 \pmod{K}\\
 0 \qquad \mbox{otherwise}.
 \end{cases}
 \end{equation}
 If $\sum_{i} u_i=0$,
 then $\mathbf{u}$ lies in the span of the eigenvectors of $Q^TQ$
corresponding to nonzero eigenvalues of $Q^TQ$.
The kernel of the positive semi-definite matrix $Q^TQ$ coincides with the kernel
of $Q$,
 which is one-dimensional since
 the characteristic polynomial of $Q$ is $(x-1)^K - (-1)^K$.
  Because $Q^TQ$ is symmetric, the eigenvectors are mutually orthogonal.
Consequently, if $\sum_i u_i=0$, then
\begin{equation}
\mathbf{u}^T (Q^TQ) \mathbf{u} \geq C_K \|\mathbf{u}\|^2,
\end{equation}
where $C_K>0$ is the first nonzero eigenvalue of $Q^TQ$.
Hence for $\rho>R_\eps$, $\|\mathbf{u}\|^2$ satisfies
\begin{equation}\label{eq:85d}
\frac{1}{2}\left( \frac{\de^2}{\de \rho^2} + \frac{1}{\rho} \frac{\de}{\de \rho} \right)
\|\mathbf{u}\|^2
\geq  C_\eps C_K \|\mathbf{u}\|^2,
\end{equation}
or written alternatively,
\begin{equation}\label{eq:85d2}
\left( -\frac{\de^2}{\de \rho^2}- \frac{1}{\rho} \frac{\de}{\de \rho}  + 2 C_\eps C_K\right)\|\mathbf{u}\|^2 \leq 0.
\end{equation}
This equation is key!
Note that all bounded solutions of
\begin{equation}
\left(-\frac{\de^2}{\de\rho^2} - \frac{1}{\rho} \frac{\de}{\de \rho} + 2 C_\eps C_K \right)f(\rho)=0
\end{equation}
are scalar multiples of the modified Bessel functions of the first kind, $K_0\left(c \rho\right)$, where $c=(2C_\eps C_K)^{-1/2}$.
The function $K_0(t)$ decays to zero
near $t=\infty$ like
\begin{equation}\label{eq:Besseldecay}
K_0\left(t\right) \sim \frac{\e^{-t}}{\sqrt{t}}.
\end{equation}
The inequality
\begin{equation}\label{eq:ineq}
\|\mathbf{u}(\rho)\|^2 \leq \eps^2 \frac{K_0\left( c \rho \right)}{K_0 \left( c R_\eps \right)}
\end{equation}
holds at $\rho=R_\eps$ and $\rho=\infty$.
 By the maximum principle (working in the coordinate $s=\frac{1}{\rho}$ on the finite interval $s \in (0, R_\eps^{-1})$),
the inequality in \eqref{eq:ineq} holds for $\rho>R_\eps$.
\end{proof}

\subsection{Fixed points as modules over \texorpdfstring{$\boldsymbol{R=\C[y,z]/(y^K-z^N)}$}{R = C[y,z]/(y^K-z^N)}} \label{sec:otherinterpretations}

In this section, which is not necessary for the rest of the paper,
we briefly discuss the relation between our results and those of 
Piontkowski in
\cite{Piontkowski}. This relation was proposed to the first author 
by Eugene Gorsky.

All of the $\C^\times$-fixed points in $\cM_{K,N}$
lie in the central fiber, which we expect to be the
compactified Jacobian of the curve $y^K = z^N$.
In turn, this compactified Jacobian has been 
studied in \cite{Piontkowski} 
in the language of rank-1 torsion-free modules over the ring $R=\C[y,z]/(y^K-z^N)$.
Here we spell out a correspondence between good filtered Higgs bundles fixed by the $\C^\times$-action on $\cM_{K,N}$ and rank-1 torsion-free $R$-modules fixed by a certain $\C^\times$-action,
appearing in Proposition 5 of \cite{Piontkowski}.

The ring $R$ can be embedded in its normalization $\C[x]$ by taking $y=x^N$, $z=x^K$, as can any rank-1 torsion-free $R$-module $M$.
Piontkowski gives a stratification 
of the compactified Jacobian by the image of $M \subset \C[x]$ under the map 
\begin{align}
 \nu: \C[x] &\rightarrow \mathbb{N}\\ \nonumber
 f(x) &\mapsto \mathrm{deg}(f).
\end{align}
The image $\nu(M)=:\Delta$ 
satisfies $\Delta + K \subset \Delta$ and $\Delta+ N \subset \Delta$. Without loss of generality, we may choose the embedding of $M$ in $\C[x]$ so that $\Delta$ contains $0$.
Proposition 3 of \cite{Piontkowski} gives a classification of all such $\Delta$.
Fixing some $\Delta$, there is a distinguished module $M_0 \in \nu^{-1}(\Delta)$. The module $M_0$ is generated by $1=x^{a_0}, x^{a_1}, \cdots, x^{a_{K-1}}$
where $a_i \equiv iN \pmod K$. 
The relations among the generators are just the relations in $\C[x]$, i.e.
\begin{equation}
 x^N \cdot x^{a_i} = x^{Kb_i} x^{a_{i+1}}.
\end{equation}
Taking $z=x^K$, we can encode this set of relations as a matrix
\begin{equation}
 x^N \cdot\begin{pmatrix} x^{a_0} \\ x^{a_1} \\ \vdots \\ x^{a_{K-1}} \end{pmatrix} =
 \begin{pmatrix}
  0 & z^{b_1} &  &  & \\
  & 0 & z^{b_2} & & \\
   & & \ddots & \ddots & \\
    & & & 0 & z^{b_{K-1}} \\
    z^{b_K} & & & & 0 
 \end{pmatrix}
\begin{pmatrix} x^{a_0} \\ x^{a_1} \\ \vdots \\ x^{a_{K-1}} \end{pmatrix} .
\end{equation}
Take $\varphi_{\mathbf{b}}/\de z$ to be this matrix,
and compare with the $\C^\times$-fixed point $(\cE_{\mathbf{b}}, \varphi_{\mathbf{b}}$) in  
 \eqref{eq:b-higgs-bundle}.
Note that the $\C^\times$-action on $\C[x]$ given by
\begin{equation}
 \rho_t(x) = \e^{\frac{-t}{K+N}} x, \qquad t \in \C/2\pi \I (K+N) \Z,
\end{equation}
fixes the module $M_0 \subset \C[x]$.

\section{A regulated \texorpdfstring{$L^2$}{L2} norm}\label{sec:mu}

In \cite{hitchin87},
Hitchin considered a real-valued function $\mu$ on the moduli space $\cM_{2}(C)$ of
harmonic $SL(2)$-bundles over a compact Riemann surface $C$.
The function $\mu$ has three main interpretations:
\begin{itemize}
\item $\mu$ is a moment map generating the $S^1$-action on $\cM_{K}(C)$, with
respect to one of its symplectic forms $\omega_I$.
\item $\mu([(\cE, \varphi)])$ is the $L^2$ norm of $\varphi$ in the harmonic metric: $\mu([(\cE, \varphi)]) =
\frac{\I}{\pi} \int_C \mathrm{Tr} \; \varphi \varphi^{\dagger_h}$.
\item If $[(\cE, \varphi)]$ is a $\C^\times$-fixed point, then
$\mu$ can be computed explicitly; there are two cases to consider:
\begin{itemize}
\item $\cE = L \oplus (L^* \otimes \Lambda^2 \cE)$
and
\begin{equation} \label{eq:mu-deg-compact}
	\mu([(\cE, \varphi)]) = \half \left(\mathrm{deg} \;L - \frac{1}{2}\right),
\end{equation}
\item $\cE$ is stable and $\mu([(\cE, \varphi)]) = 0$.
\end{itemize}
\end{itemize}

We expect that there is an analogous function $\mu$ on $\cM_{K,N}$ which
serves as a moment map for the $S^1$ action there.
In this paper we do not give a construction of this function:
rather, we restrict ourselves to some suggestive computations at the $\C^\times$-fixed
points.

For the Higgs bundles $(\cE, \varphi) \in \cC_{K,N}$, the ordinary $L^2$ norm
of $\varphi$ is infinite, because of the divergent
contribution near the irregular singularity at $z = \infty$.
However, there is a natural way of regularizing this infinity: we define
\begin{equation}\label{eq:mu1}
\mu([\cE, \varphi]) = \frac{\I}{\pi} \int \mathrm{Tr} \left( \varphi \wedge \varphi^{\dagger_h} - \Id |z|^{2N/K} \de z \de \zbar \right).
\end{equation}
This regularized $L^2$ norm turns out to be explicitly computable
at the $\C^\times$-fixed points:

\begin{prop}\label{prop:muone}
Given an ordered $K$-partition of $N$ $\mathbf{b}$, let
$(\cE_\bfb, \varphi_\bfb)$ be the associated Higgs bundle from Proposition \ref{prop:classification-fixed-points}.
Then
\begin{equation}\label{eq:mu1value}
\mu([\cE_\bfb, \varphi_\bfb]) = \frac{K}{K+N} \|B \mathbf{b}\|^2
\end{equation}
where $B$ is the matrix defined in \eqref{eq:B}.
\end{prop}

\begin{proof} We compute directly:
\begin{eqnarray}\label{eq:mu1proof}
 \mu &=& \frac{\I}{\pi} \int \mathrm{Tr}
 \left(  \varphi \wedge \varphi^{\dagger_h}
 - \Id |z|^{2N/K} \de z \de \zbar \right)\\ \nonumber
 &\overset{\eqref{eq:hexpression}}{=}& \frac{\I}{\pi} \int_{0}^{2 \pi} \int_{0}^{\infty}
|z|^{2N/K} \sum _{i=1}^K  \left(  \e^{u_{i}-u_{i+1}}-1\right) (-2 \I |z| \de |z| \wedge \de \vartheta) \\\nonumber
&=&4 \int_{0}^\infty
\sum_{i=1}^K  \left(  \e^{u_{i}-u_{i+1}}-1\right)
\de \left( \frac{K}{2(K+N)}|z|^{\frac{2(K+N)}{K}} \right) \\\nonumber
&=& 4 \left[ \sum_{i=1}^K  \left(  \e^{u_{i}-u_{i+1}}-1\right)
\frac{K}{2(K+N)}|z|^{\frac{2(K+N)}{K}} \right|_{0}^\infty \\ \nonumber
& &
- 4 \int_{0}^\infty \sum _{i=1}^K  \de \left(  \e^{u_{i}-u_{i+1}}-1\right)
\frac{K}{2(K+N)}|z|^{\frac{2(K+N)}{K}} \\ \nonumber
&\overset{\eqref{eq:bc}}{=}&
- 4 \int_{0}^\infty \frac{K}{2(K+N)}  \sum_{i=1}^K  |z|^2   \frac{\de u_i}{\de|z|} |z|^{\frac{2N}{K}} \left(  \e^{u_{i}-u_{i+1}}-\e^{u_{i-1}-u_{i}}\right)
\de|z| \\\nonumber
&\overset{\eqref{eq:ODE}}{=}&- 4 \int_{0}^\infty \frac{K}{2(K+N)}  \sum _{i=1}^K  |z|^2  \frac{\de u_i}{\de|z|} \frac{1}{4 |z|} \de \left( |z| \frac{\de u_i}{\de |z|}\right)
\\\nonumber
&=& - \frac{K}{4(K+N)} \left[\sum _{i=1}^K \left( |z| \frac{\de u_i}{\de |z|}\right) ^2 \right|_0^\infty \\\nonumber
&\overset{\eqref{eq:bc}
, \eqref{eq:btocviaB}}{=}& \frac{K}{K+N} \left\|\boldsymbol{\alpha} \right\|^2.
\end{eqnarray}
Lastly, $\boldsymbol{\alpha} = -B \mathbf{b}$.
\end{proof}
Using \eqref{eq:btocviaB} we see that this can also be written as
\begin{equation}\label{eq:mu2}
 \mu = \frac{K}{K+N}\sum_{i=1}^K \left(\pdeg \cL_i \right)^2,
 \end{equation}
a formula reminiscent of \eqref{eq:mu-deg-compact}.

\section{Fixed points and minimal models}\label{sec:Walgebra}

In this section we describe a somewhat mysterious connection between
the fixed points of the $\C^\times$-action on $\cM_{K,N}$ and certain
representations of the vertex algebra $\cW_{K}$.

\subsection{\texorpdfstring{$\cW_{K}$}{W(K)} and its minimal models}

For any $K \ge 2$, there is a ``$W$-algebra'' $\cW_K$.
$\cW_K$ is a vertex algebra containing the Virasoro vertex algebra.
$\cW_2$ is the Virasoro vertex algebra;
for the definition of $\cW_3$ see \cite{Fateev1987a},
and for $\cW_K$ see \cite{Fateev1988}.
A conceptually clean definition of $\cW_K$ uses quantum Drinfeld-Sokolov
reduction of the affine vertex algebra
$\widehat{sl(K)}$ by a principal
$sl(2) \subset sl(K)$.

For any pair $(p,q)$ with $p,q \in \Z_+$
relatively prime, there is a distinguished collection $\Lambda_{K;p,q}$
of representations of
$\cW_K$, called the ``$(p,q)$ minimal model of $\cW_K$.''
Our interest in this paper is in the $(K,K+N)$ minimal model of $\cW_K$.
Thus for convenience we write $\Lambda_{K,N} = \Lambda_{K;K,K+N}$.

\subsection{Irreducible representations of \texorpdfstring{$\mathcal{W}_K$}{W(K)}}\label{sec:Wclass}

The set $\Lambda_{K,N}$ consists of highest-weight representations of
$\cW_K$, parameterized by dominant weights $\Lambda$ of $sl(K)$.
The set of weights which occur is given as follows:

\begin{prop} \label{prop:highestweight}
There is a bijection between $\Lambda_{K,N}$ and the set of
cyclic $K$-partitions of $N$. Given an ordered $K$-partition
of $N$, $\mathbf{b}$,
the associated highest weight is
\begin{equation}
\Lambda_\bfb = P_{\mathbf{1}^\perp} \psi(\mathbf{b})
\end{equation}
where $P_{\mathbf{1}^\perp}$ denotes
orthogonal projection onto $\mathbf{1}^\perp \subset \R^K$,
with $\mathbf{1} = (1, \dots, 1)$,
and
\begin{equation}
	\psi(\bfb) = (n_1, \dots, n_K), \qquad n_i = N - \sum_{j=1}^i b_j. \label{eq:def-psi}
\end{equation}
\end{prop}
\begin{proof}
 See (6.73) of \cite{W-algebras}, or (4.71) of \cite{cordovashao}.
\end{proof}

Because $\cW_K$ contains the Virasoro algebra, each representation
in $\Lambda_{K,N}$ is in particular a representation of the Virasoro
algebra; thus it has a central charge $c \in \R$ and a Virasoro
highest weight $h \in \R$. The \ti{effective Virasoro central charge}
is defined by
\begin{equation} \label{eq:ceff}
	c_\eff = c - 24h.
\end{equation}

\begin{prop} \label{prop:ceff-formula}
The effective Virasoro central charge of
the representation of $\cW_{K}$ with highest weight $\Lambda$
is
\begin{equation}
c_{\mathrm{eff}}(\Lambda) =
K-1- \frac{12K}{K+N} \left\| \Lambda - \frac{N}{K} \rho\right\|^2
\end{equation}
where
\begin{equation}\label{eq:rho}
 \rho= \frac{1}{2}\left( K-1, K-3, \cdots, 3-K, 1-K\right)
\end{equation}
 is $\frac{1}{2}$ the sum of the positive weights of $SL(K)$.
\end{prop}

\begin{proof} See \cite{W-algebras}.
\end{proof}

\subsection{Fixed points and minimal models} \label{sec:maintheorem}

Given $[\mathbf{b}]$ a cyclic $K$-partition of $N$, we have associated two
different sorts of object, each with an associated number:
\begin{itemize}
 \item a point $[(\cE_\bfb, \varphi_\bfb)] \in \cM_{K,N}$ fixed by the $\C^\times$-action (Proposition \ref{prop:classification-fixed-points}), with a corresponding
 number $\mu([(\cE_\bfb, \varphi_\bfb)])$ (Proposition \ref{prop:muone}),
 \item $\Lambda_\bfb$, a highest weight in the $(K, K+N)$ minimal model of $\mathcal{W}_K$ (Proposition
 \ref{prop:highestweight}),
 with a corresponding number $c_{\eff}(\Lambda_\bfb)$ (Proposition \ref{prop:ceff-formula}).
\end{itemize}
Then,
\begin{thm} \label{thm:mainS1}
Let $\mathbf{b}$ be an ordered $K$-partition of $N$
and $\mu = \mu([(\cE_\bfb, \varphi_\bfb)])$,
$c_\eff= c_\eff(\Lambda_\bfb)$.
Then
\begin{equation}\label{eq:dictionary}
 \mu = \frac{1}{12} \left(K - 1 - c_\eff \right).
\end{equation}
\end{thm}

\begin{proof}
These two numbers are
\begin{align}\label{eq:twonum}
\mu &= \frac{K}{K+N} \left\|B \mathbf{b} \right\|^2, \\
c_\eff&=  K-1- \frac{12K}{K+N} \left\| P^\perp_{\mathbf{1}} \psi(\bfb)   -  \frac{N}{K} \rho \right\|^2, \label{eq:twonum-2}
\end{align}
where $P^\perp_{\mathbf{1}}$ is the orthogonal projection onto the subspace orthogonal to $\mathbf{1}$,
$\rho$ is as in \eqref{eq:rho}, $B$ is as in \eqref{eq:B}.
We rearrange \eqref{eq:twonum-2} to
\begin{equation}
 \frac{1}{12}(K-1-c_{\eff})= \frac{K}{K+N} \left\| P^\perp_{\mathbf{1}} \psi(\bfb) -  \frac{N}{K} \rho \right\|^2.
\end{equation}
This reduces \eqref{eq:dictionary} to
\begin{equation}
	\left\|P^\perp_{\mathbf{1}} \psi(\bfb) -  \frac{N}{K} \rho\right\| = \norm{B \bfb} .
\end{equation}
To prove this we will show
\begin{equation} \label{eq:matchvectors}
 P^\perp_{\mathbf{1}} \psi(\bfb)   -  \frac{N}{K} \rho =-M B \mathbf{b},
\end{equation}
where $M$ is the permutation matrix corresponding to the cyclic shift by $1$,
i.e. $M_{ij}=\delta_{i+1,j}$.
The proof follows the characterization of the matrix $B$ in \eqref{eq:B} as the unique matrix such that
\begin{equation} \label{eq:defining}
B(M-\mathrm{Id})=(M-\mathrm{Id})B=P_{\mathbf{1}^\perp}
\end{equation}
and $B\mathbf{1}=\mathbf{0}$.
We compute:
\begin{align}
 -MB \mathbf{b}&=-MB \left((M^{-1}- \mathrm{Id}) \psi(\bfb) + \begin{pmatrix} N &0 & \cdots 0 \end{pmatrix}^T
 \right)\\
 &=(M-\mathrm{Id}) B \psi(\bfb) - B M \begin{pmatrix} N &0 & \cdots 0 \end{pmatrix}^T\\
  &=P_{\mathbf{1}^\perp} \psi(\bfb) - \frac{N}{K} \rho.
\end{align}
The first line is \eqref{eq:def-psi}.
The second line follows from $BM=MB$. The third line follows from \eqref{eq:defining}
and \eqref{eq:rho}.
\end{proof}
\begin{rem}
 Note that \eqref{eq:matchvectors} gives a direct relation
 between the highest weight $\Lambda_\mathbf{b}$ and the parabolic weights $\boldsymbol{\alpha}$:
 \begin{equation}
  \left(\Lambda_{\mathbf{b}}- \frac{N}{K} \rho\right)_i=-\alpha_{i+1}.
 \end{equation}

\end{rem}

\section{The case of \texorpdfstring{$\cM_{2,3}$}{M(2,3)}} \label{sec:M23}

In this section we discuss in more detail the simplest nontrivial case
of our story, namely the case $K=2$, $N=3$.

\subsection{The stratification}

We begin by studying the stratification of $\cM_{2,3}$ by isomorphism
type of the underlying filtered bundles. This case is simple enough
that we can describe the strata, and representative Higgs bundles,
in a completely explicit way.

\begin{prop}
$\cM_{2,3}$ is decomposed into two strata:
\begin{align}
	\cM_{2,3}^{\rmsmall} &= \cM_{2,3}^{[(3,0)]} = \left\{ [(\cE,\varphi)] \in \cM_{2,3} \, \vert \, \cE \simeq \cO\left(\frac34\right) \oplus \cO\left(-\frac34\right) \right\}, \\
	\cM_{2,3}^{\rmbig} &= \cM_{2,3}^{[(2,1)]} = \left\{ [(\cE,\varphi)] \in \cM_{2,3} \, \vert \, \cE \simeq \cO\left(\frac14\right) \oplus \cO\left(-\frac14\right) \right\}.
\end{align}
\end{prop}

\begin{proof} By Proposition \ref{prop:shape} we know that $\cE = \cO(\alpha) \oplus \cO(-\alpha)$
where $\alpha = \frac14 \pmod 1$. On the other hand, by \eqref{eq:nu-shift},
$\bar\varphi$
increases the weight by $\frac32$; it follows that the gap $2 \abs{\alpha}$
cannot exceed $\frac32$. This leaves only the two possibilities listed.
\end{proof}

\begin{prop} \label{prop:stratum-a}
There is a bijection $\cM^{\rmsmall}_{2,3} \simeq \C$. For each $u \in \C$
a representative $(\cE_u, \varphi_u)$ is
\begin{equation}
	\cE_u = \cO\left(\frac34\right) \oplus \cO\left(-\frac34\right), \qquad \varphi_u = \begin{pmatrix} 0 & z^3 + u \\ 1 & 0 \end{pmatrix} \, \de z.
\end{equation}
\end{prop}

\begin{proof} Suppose given any $(\cE, \varphi) \in \cM^{\rmsmall}_{2,3}$.
Our aim is to produce a trivialization $(e_1,e_2)$ of $\cE$
in which $\varphi$ takes the
form $\varphi_u$. For this, let $e_1$ denote a nowhere vanishing
section of the (unique) filtered
line subbundle $\cL \subset \cE$ with $\cL \simeq \cO(\frac34)$.
Such a section has
$\nu_\infty(e_1) = -\frac34$. Then let
\begin{equation} \label{eq:phiaction1}
	e_2 = \bar\varphi(e_1).
\end{equation}
Since $\nu_\infty(e_1) = -\frac34$, by \eqref{eq:nu-shift} we have
$\nu_\infty(e_2) = \frac34$. In particular
$\nu_\infty(e_1) - \nu_\infty(e_2) \notin \Z$, which implies
that $e_1 \wedge e_2$ is not identically zero.
Moreover, $\nu_\infty(e_1 \wedge e_2) = 0$, and
since $\pdeg \wedge^2 \cE = 0$, it follows that $e_1 \wedge e_2$ has
no zeros. Thus $(e_1,e_2)$ indeed gives a global trivialization of $\cE$.

Relative to this trivialization, we now consider the matrix representing
$\varphi$: \eqref{eq:phiaction1} implies it is of the form
\begin{equation}
\varphi = \begin{pmatrix} 0 & P(z) \\ 1 & 0 \end{pmatrix} \de z
\end{equation}
(the zero at lower right is determined by $\Tr \varphi = 0$.)
Now by Proposition \ref{prop:hitchinbase} we know that
$\charp \, \varphi = \lambda^2 - (z^3 + u) \de z^2$ for some $u \in \C$.
This gives $P(z) = z^3 + u$.
Thus $(\cE, \varphi) \simeq (\cE_u, \varphi_u)$ as desired.

Conversely it is straightforward to check that each $[(\cE_u, \varphi_u)]$
indeed belongs to $\cM^{\rmsmall}_{2,3}$, and that they are all distinct,
since they have different characteristic polynomials.
\end{proof}

\begin{prop} \label{prop:stratum-b}
There is a bijection $\cM^{\rmbig}_{2,3} \simeq \C^2$. For each $(w,\gamma) \in \C^2$
a representative $(\cE_{w,\gamma}, \varphi_{w,\gamma})$ is
\begin{equation}
	\cE_{w,\gamma} = \cO\left(\frac14\right) \oplus \cO\left(-\frac14\right), \qquad \varphi_{w,\gamma} = \begin{pmatrix} \gamma & \frac{z^3 - w^3}{z - w} \\ z - w & -\gamma \end{pmatrix} \, \de z.
\end{equation}
\end{prop}

\begin{proof} Suppose given any $(\cE, \varphi) \in \cM^{\rmbig}_{2,3}$. As above,
let $e_1$ denote a section of $\cL \simeq \cO(\frac14)$ with $\nu_\infty(e_1) = - \frac14$.
Now consider $e_1 \wedge \bar\varphi(e_1)$. In contrast to the case of
Proposition \ref{prop:stratum-a} above, we now have $\nu_\infty(e_1 \wedge \bar\varphi(e_1)) = 1$,
so $e_1 \wedge \bar\varphi(e_1)$ vanishes at one point $w \in \C$.
Thus taking $e_2 = \bar\varphi(e_1)$ will not give a global
trivialization. Instead, we proceed as follows. Since $e_1 \wedge \bar\varphi(e_1)$
vanishes at $w \in \C$, $e_1$ is an eigenvector of $\bar\varphi$ at $w$:
in other words we have $(\bar\varphi - \gamma) e_1 = 0$ at $w$,
for some $\gamma \in \C$. Then we let
\begin{equation} \label{eq:phiaction2}
	e_2 = \frac{\bar\varphi(e_1) - \gamma e_1}{z-w}.
\end{equation}
$e_2$ is actually regular even at $z = w$ since the numerator vanishes there.
Also $\nu_\infty(e_2) = \frac14$ (since the first term has $\nu_\infty = \frac14$
and the second has $\nu_\infty = -\frac54$.)
Finally, $e_1 \wedge e_2 = \frac{e_1 \wedge \bar\varphi(e_1)}{z-w}$ is also
regular everywhere and has $\nu_\infty(e_1 \wedge e_2) = 0$, so it is constant;
thus $(e_1, e_2)$ give a global trivialization of $\cE$.

Relative to this trivialization, \eqref{eq:phiaction2} says
\begin{equation}
	\varphi = \begin{pmatrix} \gamma & P(z) \\ z-w & -\gamma \end{pmatrix} \, \de z.
\end{equation}
Imposing $\charp \varphi = \lambda^2 - (z^3 + u) \de z^2$
for some $u$ forces $P(z) = z^2 + zw + w^2$.
This shows that $(\cE, \varphi) \simeq (\cE_{w,\gamma}, \varphi_{w,\gamma})$
as desired. Moreover, $(w,\gamma)$ are determined
by $(\cE, \varphi)$: $w$
is the location of the zero of $e_1 \wedge \bar\varphi e_1$, and
$\gamma$ is the eigenvalue $\bar\varphi e_1 / e_1$ at $z = w$.
Finally, it is straightforward to check
$(\cE_{w,\gamma}, \varphi_{w,\gamma}) \in \cC_{2,3}$.
\end{proof}

\subsection{The $\C^\times$-action and its fixed points}
The $\C^\times$-action in this example also admits an explicit and simple
description: 
in the coordinates we have found above, it acts linearly on each stratum.

\begin{prop}
The $\C^\times$-action preserves the strata in $\cM_{2,3}$, and:
\begin{itemize}
\item The $\C^\times$-action on $\cM_{2,3}^{\rmsmall}$ takes
$u \mapsto \e^{\frac{6t}{5}} u$.

\item The $\C^\times$-action on $\cM^{\rmbig}_{2,3}$ takes
	$(w,\gamma) \mapsto (\e^{\frac{2t}{5}}w, \e^{\frac{3t}{5}}\gamma).$
\end{itemize}
\end{prop}
\begin{proof} Just apply the $\C^\times$-action
\eqref{eq:new-cstar-action}
to the representatives
given in Propositions \ref{prop:stratum-a}-\ref{prop:stratum-b},
then use a diagonal gauge transformation to put them back
into the representative form.
\end{proof}

In particular we see immediately that there are exactly
two $\C^\times$-fixed points in $\cM_{2,3}$, one in each stratum, corresponding
to $u = 0$ and $(w,\gamma) = (0,0)$. This is in
accord with the general classification of fixed points in \S\ref{sec:fixedpoints},
which when specialized to this case gives
\begin{align}
[\bfb] &= [(0,3)] \quad \rightsquigarrow  \quad \cE_\bfb = \cO\left(-\frac34\right) \oplus \cO\left(\frac34\right), \quad \varphi_\bfb = \begin{pmatrix} 0 & 1 \\ z^3 & 0 \end{pmatrix} \, \de z, \quad \mu = \frac{9}{20}, \\
[\bfb] &= [(1,2)] \quad \rightsquigarrow \quad \cE_\bfb = \cO\left(-\frac14\right) \oplus \cO\left(\frac14\right), \quad \varphi_\bfb = \begin{pmatrix} 0 & z \\ z^2 & 0 \end{pmatrix} \, \de z, \quad \mu = \frac{1}{20}.
\end{align}
For convenience we also included the values of $\mu$ at the two fixed points,
calculated using \eqref{eq:mu1}.
According to our discussion in \S\ref{sec:Walgebra},
these two fixed points are supposed to correspond to the two representations of
the Virasoro algebra comprising the $(2,5)$ Virasoro minimal model.
Let us verify this directly.
The $(2,5)$ Virasoro minimal model has $c = - \frac{22}{5}$
and Virasoro highest weights $h = 0, -\frac{1}{5}$, which
using \eqref{eq:ceff} gives
$c_{\eff} = -\frac{22}{5}, \frac{2}{5}$.
Next \eqref{eq:dictionary}
with $K=2$ says $c_{\eff} = 1 - 12 \mu$.
Substituting $\mu = \frac{9}{20}, \frac{1}{20}$ this indeed
matches.

\subsection{The Hitchin fibration}

From \eqref{eq:B-cond-2}, the Hitchin base $\cB_{2,3}$ is
\begin{equation}
\cB_{2,3} = \{\lambda^2 - (z^3 + u) \de z^2 \, | \, u \in \C\}.
\end{equation}
Let us consider the fiber $\pi^{-1}(u)$ for some $u \in \C$.
$\pi^{-1}(u)$ meets the stratum $\cM^{\rmbig}_{2,3}$ in the
affine cubic curve $J_u = \{\gamma^2 - w^3 = u\}$,
and meets $\cM^{\rmsmall}_{2,3}$ in a single
point. The latter plays the
role of the ``point at infinity'' of $J_u$.

Let us briefly and informally explain this point. (We could hardly do better,
since we have not rigorously constructed a topology or complex structure on
$\cM_{2,3}$.)
The $\varphi_{w,\gamma}$ do not have a limit as $(w,\gamma) \to \infty$
along $J_u$.
However, consider the ``gauge transformation''
\begin{equation}
g = \I \begin{pmatrix} \frac{\gamma}{w} & w+z \\ 0 & - \frac{w}{\gamma} \end{pmatrix}
\end{equation}
A direct computation gives
\begin{equation}
\lim_{(w,\gamma) \to \infty} g \varphi_{w,\gamma} g^{-1} = \varphi_u.
\end{equation}
Thus, using $g$ for patching, we can extend the family of Higgs bundles
$(\cE_{w,\gamma}, \varphi_{w,\gamma})$ over the point $(w,\gamma) = \infty$.
What remains is to see how the filtration at $z = \infty$ extends.
For this we consider the two sections
\begin{equation}
	e_1 = \begin{pmatrix} 0 \\ -\frac{\I}{w+z} \end{pmatrix}, \quad e_2 = \begin{pmatrix} -\I (w+z) \\ \frac{\I \gamma}{w} \end{pmatrix}.
\end{equation}
At finite $(w,\gamma)$ we can compute directly
$\nu_\infty(e_1) = -\frac34$, $\nu_\infty(e_2) = \frac34$.\footnote{Here it is important that
$e_1$ is only a \ti{meromorphic} section: it would be impossible
to have a global holomorphic section of $\cE_{w,\gamma}$ with $\nu_\infty = -\frac34$.}
On the other hand,
\begin{equation}
	g e_1 = \begin{pmatrix} 1 \\ -\frac{w}{(w+z)\gamma} \end{pmatrix}, \quad g e_2 = \begin{pmatrix} 0 \\ 1 \end{pmatrix}.
\end{equation}
Thus $g e_1$ and $g e_2$ extend to
$(w,\gamma) = \infty$ and there they become the standard basis vectors.
With this in mind we can determine a filtration $\nu_\infty$ which extends
over $w = \gamma = \infty$
by the conditions that $\nu_\infty(e_1) = -\frac34$, $\nu_\infty(e_2) = \frac34$.
Now restricting to $(w,\gamma) = \infty$ we get a filtered Higgs bundle
which is isomorphic to $(\cE_u, \varphi_u)$.
Thus this family exhibits $(\cE_u, \varphi_u)$ as the limit of the
$(\cE_{w,\gamma}, \varphi_{w,\gamma})$ as desired.

Altogether, then, each fiber $\pi^{-1}(u)$ is a projective cubic curve.
The central fiber $\pi^{-1}(u=0)$ is a cuspidal cubic, containing the
two $\C^\times$-fixed points; the cusp is the fixed point with $\mu = \frac{1}{20}$.
All other fibers $\pi^{-1}(u)$, $u \in \C^\times$, are smooth complex tori.
The stratum $\cM_{2,3}^\rmsmall$ meets the
central fiber in the fixed point $[(\cE_{u=0}, \varphi_{u=0})]$
with $\mu = \frac{9}{20}$.

Finally, it seems natural to conjecture that $\mu$ extends to a function
on the whole of $\cM_{2,3}$ which gives a 
moment map for the $U(1)$-action, 
and in every fiber $\pi^{-1}(u)$ the maximum value of
$\mu$ is attained at the point $[(\cE_u, \varphi_u)] \in \cM_{2,3}^\rmsmall$.
If this conjecture is correct, then 
schematically $\cM_{2,3}$ looks as in Figure \ref{fig:K2N3modulispace}.
\begin{figure}[h]
\begin{centering}
\includegraphics[width=4.0in]{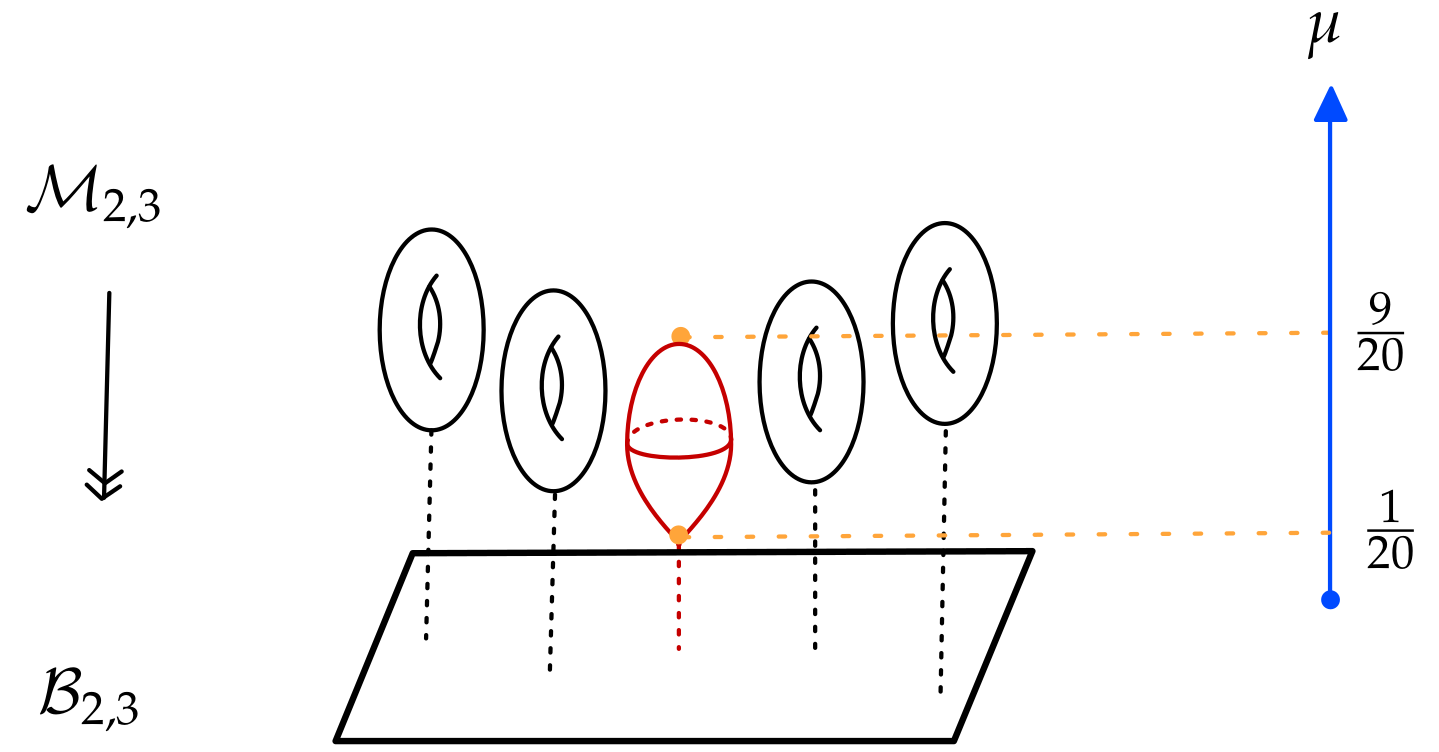}
\caption{\label{fig:K2N3modulispace}
A schematic picture of $\cM_{2,3}$.}
\end{centering}
\end{figure}

\bibliography{U1Walgebra}{}
\bibliographystyle{utphys}

\end{document}